\documentclass[12pt,leqno,fleqn]{amsart}

\usepackage{amsmath,amstext,amsthm,amssymb,amsxtra}
\usepackage{txfonts, pxfonts}

\usepackage[colorlinks,citecolor=red,pagebackref,hypertexnames=false,plainpages=false]{hyperref}
\usepackage[author-year,backrefs,msc-links,nobysame]{amsrefs}
\usepackage{mathtools} 
\mathtoolsset{showonlyrefs,showmanualtags}
\usepackage{pgf,tikz}

  \allowdisplaybreaks \swapnumbers

\theoremstyle{plain}

\newtheorem{lemma}[equation]{Lemma}
\newtheorem{proposition}[equation]{Proposition}
\newtheorem{theorem}[equation]{Theorem}
\newtheorem{corollary}[equation]{Corollary}

\newtheorem{question}[equation]{Question}

\theoremstyle{definition}
\newtheorem{definition}[equation]{Definition}

\theoremstyle{remark}

\newtheorem{remark}[equation]{Remark}

\numberwithin{equation}{section}

\def\norm#1.#2.{\lVert#1\rVert_{#2}}
\def\Norm#1.#2.{\bigl\lVert#1\bigr\rVert_{#2}}
\def\NOrm#1.#2.{\Bigl\lVert#1\Bigr\rVert_{#2}}
\def\NORm#1.#2.{\biggl\lVert#1\biggr\rVert_{#2}}
\def\NORM#1.#2.{\Biggl\lVert#1\Biggr\rVert_{#2}}

\def\ip#1,#2,{\langle #1,#2\rangle}
\def\Ip#1,#2,{\bigl\langle#1,#2\bigr\rangle}
\def\IP#1,#2,{\Bigl\langle#1,#2\Bigr\rangle}

\def\XXint#1#2#3{{\setbox0=\hbox{$#1{#2#3}{\int}$} \vcenter{\hbox{$#2#3$}}\kern-.5\wd0}}

\def\eqdef{\stackrel{\mathrm{def}}{{}={}}}

\def\Jrd{\textnormal {Jrd}}
\def\diag{\textnormal {diag}}
\DeclareMathOperator{\sgn}{sgn}
\DeclareMathOperator{\orb}{Orb}

\begin{document}

\title[Dynamics of tuples of matrices in Jordan form]{Dynamics of tuples of matrices in Jordan form}

\thanks{$^{(1)}$Partially supported by CAMGSD-LARSYS through Funda\c{c}\~{a}o para a Ci\^{e}ncia e Tecnologia (FCT/Portugal), program POCTI/FEDER.}
\author[G. Costakis]{George Costakis} \email{costakis@math.uoc.gr} \address{Department of Mathematics, University of Crete, 71409 Knossos Av., Iraklion, GREECE.}

\author[I. Parissis]{Ioannis Parissis$^{(1)}$} \email{ioannis.parissis@gmail.com} 
\address{Centro de An\'{a}lise Matem\'{a}tica, Geometria e Sistemas Din\^{a}micos, Departamento de Matem\'atica, Instituto Superior T\'ecnico, Av. Rovisco Pais, 1049-001 Lisboa, PORTUGAL}.

\begin{abstract} A tuple $(T_1,\ldots,T_k)$ of $n\times n$ matrices over $\mathbb R$ is called hypercyclic if for some $x\in\mathbb R^n$ the set $\{T_1^{m_1}T_2 ^{m_2}\cdots T_k ^{m_k}x:m_1,m_2,\ldots,m_k\in\mathbb N_0 \}$ is dense in $\mathbb R^n$.
We prove that the minimum number of $n\times n$ matrices in Jordan form over $\mathbb R$ which form a hypercyclic tuple is $n + 1$. This answers a question of Costakis, Hadjiloucas and Manoussos.
    \end{abstract}

\maketitle

\thispagestyle{empty}

\section{Introduction} \label{s.intro}

Let $X$ be a separable Banach space either over $\mathbb R$ or $\mathbb C$. Recall that a bounded linear operator $T:X\rightarrow X$ is \emph{hypercyclic} if there exists a vector $x\in X$ whose orbit $\orb(T,x)=\{x,Tx,T^2x,\ldots\}$ is dense in $X$. For a thorough study of hypercyclicity we refer to the recent book \cite{BM}. Although hypercyclicity is a phenomenon which only appears in infinite dimensions, see \cite{KI}, Feldman recently established that this is not the case if one considers more than one operator; see \cite{F}.

Following Feldman from \cite{F}, we give the following definition:

\begin{definition}\label{d.hyper} Let $T=(T_1,\ldots,T_k)$ be a $k$-tuple of commuting continuous linear operators, acting on $X$. The $k$-tuple $T$ will be called \emph{hypercyclic} if there exists a vector $x\in X$ such that the set
    \begin{equation*}
        \{T^mx:m\in\mathbb N_0 ^k\}=\{T_1^{m_1}T_2 ^{m_2}\cdots T_k ^{m_k}x:m_1,m_2,\ldots,m_k\in\mathbb N_0 \},
    \end{equation*}
    is dense in $X$. Here we use the standard multi-index notation where $m=(m_1,\ldots,m_k)\in\mathbb N_0 ^k$ and $T^m=T_1 ^{m_1}\ldots T_k ^{m_k}$.
\end{definition}

Here and throughout the paper $\mathbb N$ denotes the set of positive integers while $\mathbb N _0$ denotes the set of non-negative integers, $\mathbb N _0=\mathbb N\cup \{0\}$.

Specializing to the case $X=\mathbb R^n$ or $X=\mathbb C^n$ we have that $T$ is a $k$-tuple of commuting $n\times n$ matrices over $\mathbb R$ or $\mathbb C$ respectively. In \cite{F}, Feldman proved that in $\mathbb C^n$ there exist $(n+1)$-tuples of \emph{simultaneously diagonalizable} matrices which are hypercyclic. Furthermore, Feldman proved that there is no hypercyclic $n$-tuple of \emph{simultaneously diagonalizable} matrices acting on $\mathbb R^n$ or $\mathbb C^n$. On the other hand, in \cite{CHM}, the authors proved that on $\mathbb R^n$, $n\geq 2$, there exist $n$-tuples of \emph{non-simultaneously diagonalizable} matrices over $\mathbb R$ which are hypercyclic. For further results on hypercyclic tuples of operators in finite or infinite dimensions look at \cite{F1},\cite{F2}, \cite{KE}, \cite{JA} and \cite{CHM1}.

In this note we restrict our attention to $k$-tuples of non-simultaneously diagonalizable matrices on $\mathbb R^n$, $n\geq 2$, where every operator in the $k$-tuple $T$ is in Jordan form. In general we will write $\Jrd_{l,\gamma}$ for the Jordan block of dimension $l$ with eigenvalue $\gamma$, that is:
\begin{align*}
    \Jrd_{l,\gamma}\coloneqq
    \begin{pmatrix}
        \gamma & 1 & 0 & \ldots & 0\\
        0 &\gamma & 1 & \ddots & \vdots\\
        0 &0 &\gamma& \ddots & 0\\
        \vdots & \ddots & \ddots & \ddots & 1\\
        0 & \ldots &0 &0 &\gamma
    \end{pmatrix}
    ,
\end{align*}
where $\gamma\in\mathbb{R}$. A general $n\times n$ matrix in Jordan form over $\mathbb R$, considered here, will consist of $p$ Jordan blocks and in general can be represented as
\begin{equation*}
    J=\diag\{\Jrd_{n_1,\gamma_1},\Jrd_{n_2,\gamma_2},\ldots,\Jrd_{n_p,\gamma_p}\}=\Jrd_{n_1,\gamma_1}\oplus\Jrd_{n_2,\gamma_2}\oplus\cdots\oplus\Jrd_{n_p,\gamma_p},
\end{equation*}
where $n_1+\cdots+n_p=n$ and $\gamma_1,\ldots,\gamma_p\in\mathbb R$.

A few remarks are in order:
\begin{remark}\label{r.structure}
    Throughout the exposition we consider $k$-tuples $T=(T_1,\ldots,T_k)$ of $n\times n$ matrices in Jordan form over $\mathbb R$. In general each $T_\nu$, $1\leq \nu \leq k$, will have a different number of Jordan blocks of different dimensions. However, since we consider $k$-tuples of operators that \emph{commute}, it is not hard to see that all the matrices in the $k$-tuple must have the same form, that is, it is enough to consider the case that all the $k$ operators $T_\nu$ have $p$ Jordan blocks of dimensions $n_1,\ldots,n_p$, where $p$ and $n_1,\ldots,n_p$ do not depend on which term $\nu$ in the $k$-tuple we are considering. Bearing this in mind, each operator $T_\nu$ can be written in the form
    \begin{equation}
        \label{e.Jrdterm} T_\nu= \Jrd_{n_1 ,\gamma_\nu ^{(1)}} \oplus \cdots \oplus \Jrd_{ {n_p,\gamma_\nu ^{(p)}}},
    \end{equation}
    where the number of blocks $p$ and the corresponding dimensions $n_1,\ldots,n_p$ are fixed throughout the $k$-tuple. Thus, the real number $\gamma_\nu ^{(b)}$ is the eigenvalue of the $b$-th Jordan block in the $\nu$-th operator of the $k$-tuple.
\end{remark}
\begin{remark} Suppose for a moment that $n_1=\cdots=n_p=1$, in other words, that all the Jordan blocks in the $k$-tuple are of dimension one. Because of Remark \ref{r.structure}, this means that all the operators in the $k$-tuple are diagonal. However, this case has already been considered by Feldman in \cite{F}. We will therefore assume for the rest of the paper that $n_b>1$ for at least one block in each of the matrices of the $k$-tuple.
\end{remark}
Concerning Jordan forms in $\mathbb R^2$, the following Theorem was proved in \cite{CHM}
\begin{theorem}
    [Costakis, Hadjiloucas, Manoussos \cite{CHM}]\label{t.Jrd2x2} There exist $2\times2$ matrices $A_j$, $j=1,2,3,4$ in Jordan form over $\mathbb R$ such that $(A_1,A_2,A_3,A_4)$ is hypercyclic.
\end{theorem}

The authors in \cite{CHM} raised the following Question:

\begin{question} \label{q.chm} What is the minimum number of $2\times 2$ matrices in Jordan form over $\mathbb R$ so that their tuple is hypercyclic?
\end{question}

In Theorem \ref{t.Jrd2x2}, since the dimension is two, all the matrices in Jordan form are necessarily Jordan blocks, that is each matrix has a single eigenvalue (remember we exclude the case that one of the matrices is diagonal). It is not hard to see that the conclusion of Theorem \ref{t.Jrd2x2} is exceptional as in $\mathbb R^n$, for $n\geq 3$, there is no $k$-tuple of matrices, each one being exactly a Jordan block, which is hypercyclic:
\begin{proposition}
    \label{p.noJordanBlock} Let $n\geq 3$ and $k\in \mathbb N$. For any $\gamma_1,\ldots,\gamma_k\in\mathbb R$ consider the $k$-tuple of Jordan blocks $J=(\Jrd_{n,\gamma_1},\Jrd_{n,\gamma_2},\ldots,\Jrd_{n,\gamma_k})$. Then $J$ is not hypercyclic.
\end{proposition}

Thus in dimension $n\geq 3$ we have to consider $k$-tuples $T=(T_1,\ldots,T_k)$ where each one of the matrices $T_\nu$, $1\leq\nu\leq k$, is in Jordan form over $\mathbb R$ and consists of more than one Jordan blocks.

For $n\times n$ matrices in Jordan form over $\mathbb R$ we have the following result in the negative direction:
\begin{proposition}
    \label{p.noJordanMatrix}For $n,k\in\mathbb N$ we consider a $k$-tuple of $n\times n$ matrices in Jordan form over $\mathbb R$, $T=(T_1,\ldots,T_k)$, where each $T_\nu$ consists of $p$ Jordan blocks of dimensions $n_1,\ldots,n_p$ as in \eqref{e.Jrdterm}. \\
    (i) Suppose that $n_b\geq 3$ for at least one $b\in\{1,2,\ldots,p\}$. Then $T$ is not hypercyclic.\\
    (ii) If $k=n$ then $T$ is not hypercyclic.
\end{proposition}
\begin{remark}
    Observe that part (ii) of Proposition \ref{p.noJordanMatrix} is only interesting when all the Jordan blocks in each one of the matrices of the $n$-tuple have dimension $n_b\leq 2$. Otherwise, part (i) gives a stronger statement.
\end{remark}

The main result of this paper is the following theorem:
\begin{theorem}
    \label{t.njordan} Fix a positive integer $n\geq 2$ and let $p_1\geq 1$ and $p_2\geq 0$ be given non-negative integers such that $2p_1+p_2=n$. There exists a hypercyclic $(n+1)$-tuple of $n\times n$ matrices in Jordan form over $\mathbb R$ where each matrix in the tuple consists of $p_1$ Jordan blocks of dimension $2$ and $p_2$ Jordan blocks of dimension $1$.
\end{theorem}

\begin{remark}
    The case $n=2$ of the previous theorem was proved by M. Kolountzakis in \cite{KOL}.
\end{remark}

Our main Theorem \ref{t.njordan} together with part (ii) of Proposition \ref{p.noJordanMatrix} gives us as a corollary the answer to Question \ref{q.chm}. In fact, we answer the corresponding question in $\mathbb R^n$ for any $n\geq 2$:

\begin{corollary} The minimum number of $n\times n$ matrices in Jordan form over $\mathbb R$ which form a hypercyclic tuple is $n+1$.
\end{corollary}

The rest of the paper is organized as follows. In section \ref{s.not} we present some general guidelines and conventions concerning the notations in this paper. In section \ref{s.aux} we present some calculations which occur frequently in dealing with Jordan blocks. We are able to turn the hypercyclicity condition in a condition which is linear in $(m_1,\ldots,m_k)$. This will turn out to be much more flexible than the original definition of hypercyclicity. In section \ref{s.negative} we take advantage of this linear reformulation of the definition of hypercyclicity in order to prove the negative results contained in Propositions \ref{p.noJordanBlock} and \ref{p.noJordanMatrix}. 

Finally, section \ref{s.main} contains the proof of the main result, Theorem \ref{t.njordan}. The proof of Theorem \ref{t.njordan} relies on the linear reformulation of the problem mentioned before. In particular, we need to construct a matrix $L^+$ such that the set $\{L^+m^T:m\in\mathbb N_0 ^{n+1}\}$ is dense in $\mathbb R^n$. The entries of $L^+$ have a special structure, imposed by the fact that we consider tuples of matrices \emph{in Jordan form} over $\mathbb R$.  In Theorem \ref{t.vectorsu}, we exploit the multi-dimensional version of Kronecker's theorem in order to reduce the construction of the matrix $L^+$ to the construction of a certain set of vectors, the rows of $L^+$, which should be linearly independent over $\mathbb Q$. The construction of these vectors is done by induction in the dimension $n$ in conjunction with the solution of a non-linear equation in Lemma \ref{l.nonlinear} which guarantees that the entries of our vectors will have the desired structure.  We will take up all these issues in the final section of this paper.

\section{Acknowledgements} We would like to thank the anonymous referee for an expert reading and suggestions that helped us improve the quality of this paper.

\section{Notations}\label{s.not} A few words about the notation are necessary. In many parts of the paper the notation becomes cumbersome due to the nature of the problem. However we consistently use the same notation which we present now. In general we will consider $k$-tuples $T=(T_1,\ldots,T_k) $ of $n\times n$ matrices in Jordan form over $\mathbb R$. The $\nu$-th matrix in the $k$-tuple consists of $p$ Jordan blocks of dimensions $n_1,\ldots,n_p$ with $n_1+\cdots+n_p=n$. Each block is in turn defined by means of its dimension and a real eigenvalue. We will always use the symbol $b$, where $1\leq b \leq p$, to index the blocks. Thus a typical operator in the $k$-tuple is of the form
\begin{align*}
    T_\nu=\Jrd_{n_1,\gamma_\nu ^{(1)}}\oplus\cdots\oplus \Jrd_{n_p,\gamma_\nu ^{(p)}}=\oplus_{b=1} ^p \Jrd_{n_b,\gamma_\nu ^{(b)}}.
\end{align*}
Hopefully these general guidelines will help the reader throughout the exposition.

\section{Auxiliary Calculations}\label{s.aux} All the results in this note depend heavily on some explicit calculations. There are two types of calculations involved in the proof. The first concerns operations on single Jordan blocks, that is, powers of Jordan blocks and multiplication of powers of Jordan blocks. Such calculations appear for example in the context of Proposition \ref{p.noJordanBlock}. The second type of calculations concerns operations on matrices that consist of several Jordan blocks each. However, since every matrix in the $k$-tuple is block diagonal, all the operations we are considering here go through in each block as in the case of single Jordan blocks; each block behaves independently than the other blocks in terms of taking powers and multiplying with other Jordan matrices in the $k$-tuple, since all matrices have the same block structure.

\subsection{Operations on single Jordan blocks} First we calculate the powers of a single Jordan block:
\begin{lemma}
    \label{l.powers} For $n,m\in\mathbb N$ and $\gamma\in\mathbb R\setminus\{0\}$ we have
    \begin{align*}
        (\Jrd_{n,\gamma})^{m}=\gamma^{m}
        \begin{pmatrix}
            1 & p_{1} ^{(m)} & p_{2} ^{(m)} & \ldots & p_{n-1} ^{(m)}\\
            0 &1 & p_{1} ^{(m)} & \ddots & \vdots\\
            0 &0 &1& \ddots & p_{2} ^{(m)} \\
            \vdots & \ddots & \ddots & \ddots & p_{1} ^{(m)}\\
            0 & \ldots &0 &0 &1
        \end{pmatrix}
        .
    \end{align*}
    The real numbers $p^{(m)}_{j}$ are defined as
    \begin{eqnarray*}
        p^{(m)}_{j}=\binom{m}{j}\frac{1}{\gamma^{j}}, \, \, \, j=1,\ldots,n-1,
    \end{eqnarray*}
    with the understanding that $\binom{m}{j}=0$ whenever $m<j$.
\end{lemma}
In the next Lemma we calculate the product of powers of Jordan blocks.
\begin{lemma}
    \label{l.ktuple}Let $n,k\in\mathbb N$ and $m=(m_1,\ldots,m_k)\in\mathbb N_0 ^k$. Let $\gamma_1,\ldots,\gamma_k$ be the eigenvalues that define the Jordan blocks and set $\gamma=(\gamma_1,\ldots,\gamma_k)$. We also set $J=(\Jrd_{n,\gamma_1},\ldots,\Jrd_{n,\gamma_k})$. We then have
    \begin{equation*}
        J^m=\gamma^m
        \begin{pmatrix}
            1 & d_{1} ^{(m)} & d_{2} ^{(m)} & \ldots & d_{n-1} ^{(m)}\\
            0 &1 & d_{1} ^{(m)} & \ddots & \vdots\\
            0 &0 &1& \ddots & d_{2} ^{(m)} \\
            \vdots & \ddots & \ddots & \ddots & d_{1} ^{(m)}\\
            0 & \ldots &0 &0 &1
        \end{pmatrix}
        ,
    \end{equation*}
    where now the diagonals of $J^m$ are defined by the numbers $d^{(m)}_j$:
    \begin{align*}
        d_j ^{(m)}= \sum _{|\beta|=j}\binom{m}{\beta}\frac{1}{\gamma^\beta}= \sum_{\stackrel{\beta_1+\cdots+\beta_k=j}{0\leq\beta_p\leq j,\, p=1,\ldots,k}} \binom{m_1}{\beta_1}\cdots\binom{m_k}{\beta_k} \ \ \frac{1}{{\gamma}^{\beta_1}}\cdots\frac{1}{{\gamma} ^{\beta_k}},\quad j=1,\ldots,n-1.
    \end{align*}
\end{lemma}
The next step is to express the entries $d_1 ^{(m)},d_2 ^{(m)}$ in the first two diagonals in a simpler form. This will be enough for our purposes here. For $d_1 ^{(m)}$ we readily see that
\begin{align}
    \label{e.d1m} d_1 ^{(m)}=\sum_{\nu=1}^{k} \frac{m_\nu}{\gamma_\nu}.
\end{align}
For $d_2 ^{(m)}$ we have
\begin{align}
    \label{e.d2m} d_2 ^{(m)}=\frac{1}{2}\sum_{\nu=1} ^k \frac{m_\nu(m_\nu-1)}{\gamma_\nu ^2}+\sum_{ {1\leq \nu<\nu'\leq k}} \frac{m_\nu m_{\nu'}}{\gamma_{\nu}\gamma_{\nu'}}=\frac{1}{2}\bigg((d_1 ^{(m)})^2-\sum_{\nu=1} ^k \frac{m_\nu}{\gamma_\nu ^2}\bigg).
\end{align}
\subsection{Operations on matrices with several Jordan blocks} We now consider the general case where we have a $k$-tuple of $n\times n$ matrices in Jordan form over $\mathbb R$. As we have pointed out, the structure of each matrix should be the same for the operators to be commuting, that is, each matrix in the $k$-tuple consists of say $p$ Jordan blocks with dimensions $n_1,\ldots,n_p$, where $n_1+\cdots+n_p=n$. To fix the notation, let $T=(T_1,\ldots,T_k)$. We have that
\begin{equation}
    T_\nu=\Jrd_{n_1,\gamma_\nu ^{(1)}}\oplus\cdots\oplus \Jrd_{n_p,\gamma_\nu ^{(p)}}, \quad 1\leq \nu\leq k,
\end{equation}
where $\gamma_\nu^{(1)},\ldots,\gamma_\nu ^{(p)}\in\mathbb R$ for all $1\leq \nu \leq k$. In other words, each matrix in the $k$-tuple is a block diagonal Jordan matrix with $p$ discrete real eigenvalues. For $b\in\{1,2,\ldots,p\}$, the block $\Jrd_{n_b,\gamma_\nu^{(b)}}$ is a $n_b\times n_b$ Jordan block with eigenvalue $\gamma_\nu ^{(b)}$, where the index $\nu$ tells us which term of the $k$-tuple we are considering. We may have $n_b=1$ for some $b$'s but we exclude the possibility that $n_b=1$ for all $b\in\{1,2,\ldots,p\}$, that is, we don't allow a $k$-tuple of diagonal matrices. For $m=(m_1,\ldots,m_k)\in\mathbb N_0 ^k$, we can write
\begin{align}
    \label{e.Jordanktuple1} T^m=T_1^{m_1}\cdots T_k^{m_k}= \Pi_1 \oplus\cdots\oplus \Pi_p,
\end{align}
where each $\Pi_b$ is the tuple defined as
\begin{align}
    \label{e.Jordanktuple2} \Pi_b\eqdef \Jrd_{n_b,\gamma_1 ^{(b)}} ^{m_1}\cdots  \Jrd_{n_b,\gamma_k ^{(b)}} ^{m_k}, \quad b\in\{1,2,\ldots,p\}.
\end{align}
Observe that each block $\Pi_b $ is described by Lemma \ref{l.ktuple} with $n=n_b$ and $\gamma=(\gamma_1 ^{(b)},\ldots,\gamma_k ^{(b)})$.
\subsection{Matrices with Jordan blocks of dimension at most two} \label{ss.<3ktuple}We now turn our attention to $k$-tuples of $n\times n$ matrices in Jordan form over $\mathbb R$ where all the Jordan blocks have dimension $n_b\leq 2$. This is justified because of Proposition \ref{p.noJordanMatrix} which says that if even one of the Jordan blocks has dimension $n_b>2$ then the $k$-tuple cannot be hypercyclic. To simplify the notation let us agree that in each Jordan matrix of the $k$-tuple we have $p_1$ Jordan blocks of dimension $2$ and $p_2$ Jordan blocks of dimension $1$. Assume that in the $\nu$-th term the $2\times2$ blocks are defined by the non-zero eigenvalues $\gamma_\nu ^{(1)},\ldots,\gamma_\nu ^{(p_1)}$ and the $1\times 1$ blocks are defined by the non-zero eigenvalues $c_\nu ^{(1)},\ldots,c_\nu ^{(p_2)}$. We also write $\gamma^{(b)}=(\gamma_1 ^{(b)},\ldots,\gamma_k ^{(b)})$ for $1\leq b\leq p_1$ and $c^{(b)}=(c_1 ^{(b)},\ldots,c_k ^{(b)})$ for $1\leq b \leq p_2$. We define the matrix
$\Gamma=\{\gamma_\nu ^{(b)}\} \in \mathbb R ^{p_1\times k}$ whose rows are the vectors $\gamma ^{(b)}$, $1\leq b \leq p_1$. Similarly, the matrix $C=\{c_\nu ^{(b)}\}\in\mathbb R^{p_2\times k}$ is the matrix whose rows are the vectors $c ^{(b)}$ for $1\leq b \leq p_2$. Of course we have $2p_1+p_2=n$. Finally, the following notation will be useful. For $\Gamma$ and $C$ as before and $m=(m_1,\ldots,m_k)\in\mathbb N_0 ^k$, we consider the vector
in $\mathbb R^n$:
\begin{align*}
    V(m,\Gamma,C)\eqdef \bigg( (\gamma ^{(1)})^{m} ,\sum_{\nu=1} ^k\frac{m_\nu}{\gamma_\nu ^{(1)}} , \ldots, ( \gamma ^{(p_1)})^{m} ,\sum_{\nu=1} ^k\frac{m_\nu}{\gamma_\nu ^{(p_1)}}, (c^{(1)})^m,\ldots,(c^{(p_2)})^m \bigg).
\end{align*}

Let $T=(T_1,\ldots,T_k)$ be a $k$-tuple of $n\times n$ matrices in Jordan form over $\mathbb R$, satisfying the previous assumptions. For $m=(m_1,\ldots,m_k)\in\mathbb N_0 ^k$, $T^m$ will be given by a form similar to \eqref{e.Jordanktuple1}:
\begin{align*}
    T^m=T_1^{m_1}\cdots T_k^{m_k}= P_1 \oplus\cdots\oplus P_{p_1} \oplus {P'}_1  \oplus\ldots\oplus {P'}_{p_2} .
\end{align*}
The tuples $P_b$ are defined as:
\begin{align}
    P_b\eqdef  \Jrd_{2,\gamma_1 ^{(b)}} ^{m_1}\cdots \Jrd_{2,\gamma_k ^{(b)}} ^{m_k}, \quad b\in\{1,2,\ldots,p_1\},
\end{align}
while
\begin{align}
    P'_b\eqdef  \Jrd_{1,c_1 ^{(b)}} ^{m_1}\cdots \Jrd_{1,c_k ^{(b)}} ^{m_k}=(c^{(b)} )^{m}, \quad b\in\{1,2,\ldots,p_2\}.
\end{align}

With these notations and assumptions taken as understood, we use Lemma \ref{l.ktuple} together with the expression \eqref{e.d1m} to get a more handy characterization of hypercyclicity in the special case we are considering.

\begin{lemma}
    \label{l.<3hyper} Let $T=(T_1,\ldots,T_k)$ be a $k$-tuple of $n\times n$ matrices in Jordan form over $\mathbb R$, defined by means of the matrices $\Gamma$ and $C$. We assume that all the Jordan blocks in $T$ have dimension at most two. Then $T$ is hypercyclic if and only if the set
    \begin{align*}
        \bigg\{ V(m,\Gamma,C): m=(m_1,\ldots,m_k)\in\mathbb N_0 ^k \bigg\} ,
    \end{align*}
    is dense in $\mathbb R^n$.
\end{lemma}	

\begin{proof} Let $T=(T_1,\ldots,T_k)$ be a $k$-tuple of $n\times n$ matrices over $\mathbb R$ where each matrix in the tuple has $p_1$ Jordan blocks of dimension $2$ and $p_2$ Jordan blocks of dimension $1$. For any $m\in\mathbb N_0 ^k$ and $y\in\mathbb R^n$ a straightforward calculation using Lemma \ref{l.ktuple} yields
\begin{align}\label{e.reduced}
	(T^m y)^T=\begin{pmatrix}
	(\gamma^{(1)})^m\big(y_1+\sum_{\nu=1} ^k\frac{m_\nu}{\gamma_\nu ^{(1)}} y_2\big )\\
	(\gamma^{(1)})^m y_2 \\
	\vdots\\
	(\gamma^{(p_1)})^m\big(y_{2p_1-1}+\sum_{\nu=1} ^k\frac{m_\nu}{\gamma_\nu ^{(p_1)}} y_{2p_1}\big )\\
	(\gamma^{(p_1)})^m y_{2p_1}\\
	(c^{(1)})^m y_{2p_1+1}\\
	\vdots\\
	(c^{(p_1)})^m y_{2p_1+p_2}
    \end{pmatrix}.
\end{align}

Assume now that $T$ is a hypercyclic tuple. There exists $y=(y_1,\ldots,y_n)\in\mathbb R^n$ such that
\begin{align}\label{e.dense}
	\overline{\{ T^m y:m\in\mathbb N_0 ^k\} }=\mathbb R^n.
\end{align}
By \eqref{e.reduced} and \eqref{e.dense} we conclude that $y_j\neq 0$ for all $j\in\{2,4,6,\ldots,2p_1,2p_1+1,2p_1+2,\ldots,2p_1+p_2\}$. Now let $x\in\mathbb R^n$ be a vector with all of its entries different than zero. We define the vector $z=(z_1,\ldots,z_n)\in\mathbb R^n$ by defining its coordinates:
\begin{align*}
 z_{2b-1}& \eqdef x_{2b}y_{2b-1}+x_{2b-1}x_{2b}y_{2b},\quad\mbox{if}\quad b\in\{1,2,\ldots,p_1\},\\
 z_{2b}&\eqdef x_{2b}y_{2b},\quad\mbox{if}\quad b\in\{1,2,\ldots,p_1\},\\
 z_{p_1+b} &\eqdef x_{p_1+b} y_{p_1+b},\quad\mbox{if}\quad b\in\{1,2,\ldots,p_2\}.				
\end{align*}
Since the $k$-tuple $T$ is hypercyclic, there exists a sequence $\{m^{(\tau)}\}_{\tau \in\mathbb N}\subset \mathbb N_0 ^k$ such that $T^{m^{(\tau)}}y\to z$ in $\mathbb R^n$ as $\tau \to +\infty$. By the definition of the vector $z$ and \eqref{e.reduced} we get
\begin{align*}
\lim_{\tau\to+\infty}(\gamma^{(b)})^{m^{(\tau)}}= x_{2b}\quad\mbox{for all}\quad b=1,2,\ldots,p_1,\\
\lim_{\tau\to+\infty}(c^{(b)})^{m^{(\tau)}}= x_{2p_1+b}\quad\mbox{for all}\quad b=1,2,\ldots,p_2,
\end{align*}
and
\begin{align*}
\lim_{\tau\to+\infty} (\gamma^{(b)})^{m^{(\tau)}}(y_{2b-1}+\sum_{\nu=1} ^k\frac{m_n ^{(\tau)}}{\gamma_\nu ^{(b)}}y_{2b})=x_{2b}y_{2b-1}+x_{2b-1}x_{2b}y_{2b},\quad\mbox{for all}\quad b=1,2,\ldots,p_1.	
\end{align*}
Combining the previous convergence relations we conclude that
$$\lim_{\tau\to+\infty} \sum_{\nu=1} ^k \frac{m_n ^{(\tau)} } {\gamma_\nu ^{(b)} } = x_{2b-1}\quad\mbox{for all}\quad b=1,2,\ldots,p_1.$$
Observe that in order to conclude the previous results we had to divide by entries of $x$ or $y$ but this is justified since we have made sure that these entries are non-zero.

We have showed that if $T$ is hypercyclic then for every $x\in \mathbb R^n$ with all of its entries different than zero there exists a sequence $\{m^{(\tau)}\}_{\tau \in\mathbb N}\subset \mathbb N_0 ^k$ such that
$$V(m^{(\tau)},\Gamma,C)\to x \quad\mbox{as}\quad \tau\to+\infty.$$
Since the set 
$$\{x=(x_1,x_2,\ldots,x_n)\in\mathbb R^n:x_j\neq 0\quad\mbox{for all}\quad j=1,2,\ldots n\}$$
is dense in $\mathbb R^n$, this concludes one direction of the equivalence in the lemma. 

The opposite directions is very easy. Choose $w\in \mathbb R^n$ with $w_1=w_3=\cdots=w_{2p_1-1}=0$ and $w_2=w_4=\cdots=w_{2p_2}=w_{2p_2+1}=\cdots=w_n=1$. If $x\in\mathbb R^n$ has all of its entries different than zero we chose $\{m^{(\tau)}\}_{\tau \in\mathbb N}\subset \mathbb N_0 ^k$ such that
\begin{align*}
	(\gamma^{(b)})^{m^{(\tau)}}&\to x_{2b}\quad \mbox{as}\quad\tau\to+\infty\quad\mbox{for all}\quad b=1,2,\ldots,p_1,\\
	\sum_{\nu=1} ^k\frac{m_n ^{(\tau)}}{\gamma_\nu ^{(b)}}&\to x_{2b-1}/x_{2b}\quad \mbox{as}\quad\tau\to+\infty\quad\mbox{for all}\quad b=1,2,\ldots,p_1,\\
	(c^{(b)})^{m^{(\tau)}}&\to x_b\quad\mbox{as}\quad\tau\to+\infty\quad\mbox{for all}\quad b=1,2,\ldots,p_2.
\end{align*}
This is always possible by our hypothesis. Using \eqref{e.reduced} it is easy to see that $T^{m^{(\tau)}}w\to x$ as $\tau\to+\infty$. It readily follows that $w$ is a hypercyclic vector for $T$.
\end{proof}
A slight variant helps us write this in linear form in terms of $m\in\mathbb N_0 ^k$.

\begin{corollary}
    \label{c.<3hyper} Let $T=(T_1,\ldots,T_k)$ be a $k$-tuple of $n\times n$ matrices in Jordan form over $\mathbb R$ where all the Jordan blocks have dimension at most two. We define the $n\times k$ matrix
    \begin{align}
        \label{e.Lmatrix} L=
        \begin{pmatrix}
            \log|\gamma_1 ^{(1)}| & \log|\gamma_2 ^{(1)}| & \cdots &\log|\gamma_k^{(1)}| \\
            {1}/{\gamma_1 ^{(1)}} & {1}/{\gamma_2 ^{(1)}} & \cdots & {1}/{\gamma_k ^{(1)}} \\
            \vdots & \vdots & \cdots &\vdots \\
            \log|\gamma_1 ^{(p_1)}| & \log|\gamma_2 ^{(p_1)}| & \cdots &\log|\gamma_k ^{(p_1)}| \\
            {1}/{\gamma_1 ^{(p_1)}} & {1}/{\gamma_2 ^{(p_1)}} & \cdots & {1}/{\gamma_k ^{(p_1)}} \\
            \log|c_1 ^{(1)}| & \log|c_2 ^{(1)}| & \cdots &\log|c_k ^{(1)}| \\
            \vdots &\vdots &\cdots &\vdots\\
            \log|c_1 ^{(p_2)}| & \log|c_2 ^{(p_2)}| & \cdots &\log|c_k ^{(p_2)}|
        \end{pmatrix}
        .
    \end{align}
    If $T$ is hypercyclic then the set $\{Lm^T:m\in\mathbb N_0 ^k\}$ is dense in $\mathbb R^n$.
\end{corollary}
Here $m^T$ denotes the transpose of the vector $m=(m_1,\ldots,m_k)$.
\begin{proof} Let us denote by $V^+(m,\Gamma,C)$ the vector
	\begin{align*}
	    V^+(m,\Gamma,C)\eqdef \bigg( |(\gamma ^{(1)})^{m} |,\sum_{\nu=1} ^k\frac{m_\nu}{\gamma_\nu ^{(1)}} , \ldots, |( \gamma ^{(p_1)})^{m}| ,\sum_{\nu=1} ^k\frac{m_\nu}{\gamma_\nu ^{(p_1)}}, |(c^{(1)})^m|,\ldots,|(c^{(p_2)})^m| \bigg).
	\end{align*}
Since $T$ is hypercyclic, Lemma \ref{l.<3hyper} implies that the set     \begin{align*}
        \bigg\{ V(m,\Gamma,C): m=(m_1,\ldots,m_k)\in\mathbb N_0 ^k \bigg\}
    \end{align*}
is dense in $\mathbb R^n$ and, thus, that the set
\begin{align*}
    \bigg\{ V^+(m,\Gamma,C): m=(m_1,\ldots,m_k)\in\mathbb N_0 ^k \bigg\}
\end{align*}
is dense in $(\mathbb R^+\times \mathbb R)^{p_1}\times (\mathbb R^+)^{p_2}$. Now let $x\in\mathbb R^n $ and define the vector
$$y\eqdef(e^{x_1},x_2,e^{x_3},x_4,\ldots,e^{x_{2p_1-1}},x_{2p_1},e^{x_{2p_1+1}},e^{x_{2p_1+2}}\ldots,e^{x_{2p_1+p_2}}).$$
Since $y\in (\mathbb R^+\times \mathbb R)^{p_1}\times (\mathbb R^+)^{p_2}$ there exists a sequence $\{m^{(\tau)}\}_{\tau \in\mathbb N}\subset \mathbb N_0 ^k$ such that
$$V^+(m^{(\tau)},\Gamma,C)\to y\quad\mbox{as}\quad \tau\to+\infty.$$
This convergence is equivalent to
\begin{align*}
    |\gamma_1 ^{(b)}|^{m_1 ^{(\tau)}}\cdots |\gamma_k ^{(b)}|^{m_k ^{(\tau)}}& \to e^{x_{2b-1}}\quad\mbox{as}\quad \tau\to+\infty\quad\mbox{for all}\quad b=1,2,\ldots,p_1,\\	
	\frac{m_1 ^{(\tau)}}{\gamma_1 ^{(b)}}+\cdots+\frac{m_k ^{(\tau)}}{\gamma_k ^{(b)}} &\to x_{2b}\quad\mbox{as}\quad \tau\to+\infty\quad\mbox{for all}\quad b=1,2,\ldots,p_1,\\
	|c_1 ^{(b)}|^{m_1 ^{(\tau)}}\cdots |c_k ^{(b)}|^{m_k ^{(\tau)}} &\to e^{x_{2p_1+b}}\quad\mbox{as}\quad \tau\to+\infty\quad\mbox{for all}\quad b=1,2,\ldots,p_2.
\end{align*}
Taking logarithms, the previous three convergence relations are equivalent to

\begin{align}\label{e.conv}
  m_1 ^{(\tau)} \log|\gamma_1 ^{(b)}|+\cdots+m_k ^{(\tau)}\log |\gamma_k ^{(b)}|&\to x_{2b-1}\quad\mbox{as}\quad \tau \to+\infty\quad\mbox{for all}\quad b=1,2,\ldots,p_1,\notag\\	
	\frac{m_1 ^{(\tau)}}{\gamma_1 ^{(b)}}+\cdots+\frac{m_k ^{(\tau)}}{\gamma_k ^{(b)}} &\to x_{2b}\quad\mbox{as}\quad \tau\to+\infty\quad\mbox{for all}\quad b=1,2,\ldots,p_1,\\
	m_1 ^{(\tau)} \log|c_1 ^{(b)}|+\cdots+m_k ^{(\tau)}\log |c_k ^{(b)}| &\to x_{2p_1+ b}\quad\mbox{as}\quad \tau \to+\infty\quad\mbox{for all}\quad b=1,2,\ldots,p_2.\notag
\end{align}

Since $x\in\mathbb R^n$ was arbitrary, gathering the convergence relations \eqref{e.conv} in matrix form gives $L(m^{(\tau)})^T\to x$ as $\tau\to +\infty$. Since $x\in\mathbb R^n$ was arbitrary this concludes the proof of the lemma.
\end{proof}
\section{Negative Results}\label{s.negative} In this section we prove Propositions \ref{p.noJordanBlock} and \ref{p.noJordanMatrix}
\begin{proof}
    [Proof of Proposition \ref{p.noJordanBlock}] Let $n,k\in\mathbb N$ and $\gamma=(\gamma_1,\ldots,\gamma_k)\in\mathbb R^k$ be the eigenvalues defining a $k$-tuple of Jordan blocks $J=(J_{n,\gamma_1},\ldots,J_{n,\gamma_k})$. Now suppose $J$ is hypercyclic, that is, there exists a $x\in\mathbb R^n$ such that the set $\{J^m x:m\in\mathbb N_0 ^k\}$ is dense in $\mathbb R^n$.

    Let $J_3(m) $ be the $3\times 3$ submatrix of $J^m$ that arises from $J^m$ by deleting the first $n-3$ rows and the first $n-3$ columns, that is
    \begin{align}
        J_3(m)=\gamma^{m}
        \begin{pmatrix}
            1 & d_1 ^{(m)} & d_2 ^{(m)} \\
            0 & 1 & d_1 ^{(m)} \\
            0 & 0 & 1
        \end{pmatrix}
        .
    \end{align}
Since $\{J^m x:m\in\mathbb N_0 ^k\}$ is dense in $\mathbb R^n$ there exists a $y=(y_1,y_2,y_3)\in\mathbb R^3$ such that the set $\{J_3(m) y:m\in\mathbb N_0 ^k\}$ is dense in $\mathbb R^3$. In particular the set $\{\gamma^m y_3:m\in\mathbb N_0 ^k\}$ is dense in $\mathbb R$ so we must have $y_3\neq 0$. Now we let $w=(y_1+y_2+y_3,y_2+y_3,y_3)$ and we choose a sequence $m=m^{(\tau)}=(m_1 ^{(\tau)},m_2 ^{(\tau)},m_3^{(\tau)})$ such that $J_3 ^{m^{(\tau)}} y\rightarrow w$ as $\tau\rightarrow \infty$. We will suppress $\tau$ to simplify notation. Since $y_3\neq0$ we conclude that $\gamma^{m}\rightarrow 1$ as $\tau\rightarrow \infty$. Next we have that $\gamma^{m}(y_2+y_3d_1^{(m)})\rightarrow y_2+y_3$ as $\tau\rightarrow \infty$. We conclude that $ d_1^{(m)} \rightarrow 1 $ as $\tau\rightarrow \infty$. Finally, from the first row of $J_3 ^m$ we get that $\gamma^m(y_1+d_1^{(m)}y_2+d_2 ^{(m)}y_3)\rightarrow y_1+y_2+y_3 $ as $\tau\rightarrow \infty$. Recalling the formula for $d_2 ^{(m)}$ in equation \eqref{e.d2m} we can rewrite this as
    \begin{align}
        \label{e.3lim} \gamma^m\big(y_1+d_1^{(m)}y_2+\frac{1}{2}((d_1 ^{(m)})^2-\sum_{j=1} ^k \frac{m_j}{\gamma_j ^2})y_3\big)\rightarrow y_1+y_2+y_3 \ \mbox{as} \ \tau\rightarrow\infty.
    \end{align}
    Let us write $\ell=\lim_{\tau\rightarrow \infty} \sum_{j=1} ^k\frac{m_j}{\gamma_j ^2}$ which obviously exists. From \eqref{e.3lim} we then get that
    \begin{align*}
        y_1+y_2+\frac{1}{2} y_3 -\frac{1}{2}\ell y_3= y_1+y_2+y_3.
    \end{align*}
    But this means that $\ell=-1$ which is clearly impossible since $\sum_{j=1} ^k\frac{m_j}{\gamma_j ^2}\geq 0$ for all $m\in\mathbb N_0 ^k$.
\end{proof}
We now give the proof of the more general result for $k$-tuples of $n\times n$ matrices in Jordan form over $\mathbb R$.
\begin{proof}
    [Proof of Proposition \ref{p.noJordanMatrix}] Let us assume that a $k$-tuple $T$ of matrices in Jordan form over $\mathbb{R}$ is hypercyclic. For $m\in\mathbb N_0 ^k$ the matrix $T^m$ has the form given by \eqref{e.Jordanktuple1} and \eqref{e.Jordanktuple2}. For (i) let $\Pi_b$ be the block that has dimension $n=n_b\geq 3$. Let this block be defined by the real numbers $\gamma_1 ^{(b)},\ldots,\gamma_k ^{(b)}$. Fixing this $b$, we just write $\gamma=(\gamma_1,\ldots,\gamma_k)$. Equation \eqref{e.Jordanktuple2} shows that $\Pi_b$ will be of the form
    \begin{align}
        \Pi_b=\Jrd_{n_b,\gamma_1} ^{m_1}\cdots \Jrd_{n_b,\gamma_k} ^{m_k}.
    \end{align}
    Sine $T$ is hypercyclic and $T^m$ is a block diagonal matrix, we conclude that there exists a $y\in\mathbb R^{n_b}$ such that $\{\Pi_b y:m\in \mathbb N_0 ^k\}$ is dense in $\mathbb R^{n_b}$. Since $n_b\geq 3$ this contradicts Proposition \ref{p.noJordanBlock} so we are done.

    For part (ii) of the Proposition we consider $n$-tuples of $n\times n$ matrices in Jordan form, $T=(T_1,\ldots,T_n)$, where each one of the matrices $T_\nu$ consists of $p$ Jordan blocks of dimension $n_b\leq 2$ for all $b\in\{ 1,2,\ldots,p \}$. We adopt the notations from paragraph \ref{ss.<3ktuple}. Using Corollary \ref{c.<3hyper} we see that if $T$ is hypercyclic then $\overline{\{Lm^T:m=(m_1,\ldots,m_n)\in\mathbb N_0 ^n\}}=\mathbb R^n.$ But this means that the operator $L:\mathbb R^n\rightarrow \mathbb R^n$ has dense range and therefore is onto.  We conclude that $L$ is invertible so we must have $\overline{\mathbb N_0 ^n}=\mathbb R^n$, a contradiction.
\end{proof}

\begin{remark} In part (i) of Proposition \ref{p.noJordanMatrix} we show that if at least one of the Jordan blocks in the tuple has dimension $n_b\geq 3$ then no $k$-tuple is hypercyclic. However, the proof given above works equally well to give a stronger statement, namely that the tuple $T$ is not even \emph{somewhere dense}: for every $x\in\mathbb R^n$, the closure of the set $\{T^mx:m\in\mathbb N_0 ^k\}$ \emph{does not} contain any open balls.
\end{remark}
\begin{remark} Likewise, the proof of part (ii) of Proposition \ref{p.noJordanMatrix} gives the stronger statement that an $n$-tuple of $n\times n$ matrices in Jordan form over $\mathbb R$ is never somewhere dense. Indeed, if the orbit of the $n$-tuple $T$ is somewhere dense for some $x$ in $\mathbb R^n$ then there is a ball $B$ inside the set $L(\mathbb R^n)$. Then the set $L(\mathbb R^n)$, which is a linear subspace of $\mathbb R^n$, has necessarily dimension $n$. We conclude that $L(\mathbb R^n)=\mathbb R^n$ so that the matrix $L$ is invertible and then we proceed as in the proof above. 
\end{remark}
\section{Hypercyclic tuples of matrices in Jordan form}\label{s.main} In this section we give the proof of Theorem \ref{t.njordan}. For this we need to construct $(n+1)$-tuples of $n\times n$ matrices in Jordan form over $\mathbb R$ which are hypercyclic. For technical reasons we need to consider the two-dimensional case separately than the $n$-dimensional case for $n\geq 3$. We first give the proof in $\mathbb R^n$ for $n\geq 3$ which already contains all the ideas.

\subsection{The proof in the case $n\geq 3$} 
\label{sub.n>2} We recall that each matrix in the tuple we want to construct will consist of $p_1$ Jordan blocks of dimension $2$ and $p_2$ blocks of dimension $1$. Thus we necessarily have $2p_1+p_2=n$. Since $n\geq 3$ in the case we considering and $p_1,p_2\in\mathbb N_0$ we conclude that $p_1+p_2\geq 2$. We consider the vectors
\begin{align}\label{e.gamma}
\notag  \gamma ^{(1)} & \eqdef \begin{pmatrix} -a_1 ^{(1)}, & a_2 ^{(1)} ,& a_3 ^{(1)}, & \ldots,&a_{p_1-1} ^{(1)},& a_{p_1} ^{(1)},& a_{p_1+1} ^{(1)}, & \ldots ,&  a_{n}^{(1)},&-a_{n+1} ^{(1)} \end{pmatrix} ,\\
\notag  \gamma ^{(2)} &\eqdef\begin{pmatrix} a_1 ^{(2)},&-a_2 ^{(2)},&a_3 ^{(2)},& \ldots,& a_{p_1-1} ^{(2)},&a_{p_1} ^{(2)},&a_{p_1+1} ^{(2)},&\ldots,a_{n} ^{(2)},&-a_{n+1} ^{(2)} \end{pmatrix},\\ & \vdots \\
    \gamma ^{(p_1-1)}& \eqdef \begin{pmatrix} a_1 ^{(p_1-1)},&a_2 ^{(p_1-1)},& a_3 ^{(p_1-1)},&\ldots,-a_{p_1-1} ^{(p_1-1)},&a_{p_1}^{(p_1-1)},&a_{p_1+1} ^{(p_1-1)},&\ldots,& a_{n} ^{(p_1-1)},&-a_{n+1} ^{(p_1-1)}\end{pmatrix} ,\\
\notag  \gamma ^{(p_1)} & \eqdef \begin{pmatrix} a_1 ^{(p_1)},& a_2 ^{(p_1)}, & a_3 ^{(p_1)},&\ldots ,& a_{p_1-1} ^{(p_1)},& -a_{p_1}^{(p_1)},& a_{p_1+1} ^{(p_1)},& \ldots,& a_{n} ^{(p_1)},& -a_{n+1} ^{(p_1)}\end{pmatrix} ,
    \end{align}
where $a_\nu ^{(b)}>0$ for all $1\leq \nu \leq {n+1}$ and $1\leq b\leq p_1$. Similarly let us define
\begin{align}\label{e.c}
\notag c ^{(1)} &\eqdef\begin{pmatrix}-\delta_1 ^{(1)},&\delta_2 ^{(1)},&\delta_3 ^{(1)},&\ldots,&\delta_{p_2-1} ^{(1)},& \delta_{p_2} ^{(1)},&\delta_{p_2+1} ^{(1)},&\ldots,&  \delta_{n+1} ^{(1)}\end{pmatrix} ,\\
\notag c ^{(2)} &\eqdef\begin{pmatrix} \delta_1 ^{(2)},&-\delta_2 ^{(2)},&\delta_3 ^{(2)},&\ldots,&\delta_{p_2-1} ^{(2)},&\delta_{p_2} ^{(2)},&\delta_{p_2+1} ^{(2)},&\ldots,&\delta_{n+1} ^{(2)}\end{pmatrix} ,\\
&\vdots\\
\notag c ^{(p_2-1)} &\eqdef\begin{pmatrix} \delta_1 ^{(p_2-1)},&\delta_2 ^{(p_2-1)},&\delta_3 ^{(p_2-1)},&\ldots,&-\delta_{p_2-1} ^{(p_2-1)},&\delta_{p_2} ^{(p_2-1)},&\delta_{p_2+1} ^{(p_2-1)},&\ldots,&\delta_{n+1} ^{(p_2-1)} \end{pmatrix},\\
\notag c ^{(p_2)} &\eqdef \begin{pmatrix} \delta_1 ^{(p_2)},& \delta_2 ^{(p_2)},&\delta_3 ^{(p_2)},&\ldots,&	\delta_{p_2-1} ^{(p_2)},&-\delta_{p_2} ^{(p_2)},&\delta_{p_2+1} ^{(p_2)},&\ldots,&\delta_{n+1} ^{(p_2)} \end{pmatrix},
\end{align}
where $\delta_\nu ^{(b)}>0$ for all $1\leq \nu \leq {n+1}$ and $1\leq b \leq p_2$. 

Now we define the $(n+1)$-tuple $T=(T_1,\ldots,T_{n+1})$ by setting
\begin{align}\label{e.Thyper}
	T_\nu\eqdef \Jrd_{2,\gamma_\nu ^{(1)}}\oplus\cdots\oplus \Jrd_{2,\gamma_\nu ^{(p_1)}}\oplus\Jrd_{1,c_\nu ^{(1)}}\cdots\oplus \Jrd_{1,c_\nu ^{(p_2)}},\quad 1\leq \nu\leq n+1.
\end{align}
Recall that $\Gamma=\{\gamma_\nu ^{(b)}\}$ and $C=\{c_\nu ^{(b)}\}$. The rest of this section is devoted to defining the matrices $\Gamma$ and $C$ appropriately so that the resulting tuple $T$ defined by \eqref{e.Thyper} is hypercyclic. We will henceforth just write $T$ with the understanding that whenever $\Gamma$ and $C$ are given matrices, $T$ is defined by \eqref{e.Thyper}.

According to Lemma \ref{l.<3hyper}, the $(n+1)$-tuple $T$ is hypercyclic if and only if we have that $\overline{\{V(m,\Gamma,C):m\in\mathbb N_0 ^{n+1}\}}=\mathbb R^n$. The following Lemma will help us simplify this statement:
\begin{lemma}
    \label{l.positive} Let the matrices $\Gamma$ and $C$ be defined by \eqref{e.gamma} and \eqref{e.c} respectively. Suppose that the set $\{V(2m,\Gamma, C:m \in\mathbb N_0 ^{n+1}\}$ is dense in $(\mathbb R^+ \times\mathbb R)^{p_1}\times \mathbb (\mathbb R^+)^{p_2}$, where  $2m\eqdef (2m_1,\ldots,2m_{n+1})$. Then $\{V(m,\Gamma, C) :m\in\mathbb N_0 ^{n+1}\}$ is dense in $\mathbb R^{2p_1+p_2}=\mathbb R^n$.
\end{lemma}

\begin{proof} Let $x=(x_1,\ldots,x_n)\in\mathbb R^n$ be given. We need to approximate $x$ with vectors of the form $V(m,\Gamma,C)$ for a suitable sequence $m=(m_1,\ldots,m_{n+1})\in\mathbb N_0 ^{n+1}$. Without loss of generality we can assume that $x_j\neq 0$ for all $1\leq j \leq n$. We define the vector $\sigma=(\sigma_1,\ldots\sigma_{n+1})\in\mathbb N_0 ^{n+1}$ as
\begin{align*}
    \sigma_\nu \eqdef \begin{cases} \frac{1-\sgn(x_\nu)}{2},\quad &\mbox{if}\quad 1\leq \nu \leq n,\\ 0 ,\quad & \mbox{if}\quad  \nu= {n+1}.\end{cases}
\end{align*}
We claim that there exists a sequence $m=m^{(\tau)}$ such that $V(2m^{(\tau)}+\sigma, \Gamma, C)\rightarrow x$, as $\tau \rightarrow \infty$. Indeed we have that
    \begin{align*}
        V(2m+\sigma,\Gamma,C)=\bigg( &( \gamma ^{(1)})^{2m} ( \gamma ^{(1)})^\sigma ,\sum_{\nu=1} ^{n+1}\frac{2m_\nu}{ \gamma_\nu ^{(1)}}+\sum_{\nu=1} ^{n+1}\frac{\sigma_\nu}{ \gamma_\nu ^{(1)}} ,\\&\vdots\\& (  \gamma ^{(p_1)})^{2m} ( \gamma ^{(p_1)})^{\sigma},\sum_{\nu=1} ^{n+1}\frac{2m_\nu}{ \gamma_\nu ^{(p_1)}}+\sum_{\nu=1} ^{n+1}\frac{\sigma_\nu}{ \gamma_\nu ^{(p_1)}},\\& ( c^{(1)})^{2m}( c^{(1)})^{\sigma},\ldots,( c^{(p_2)})^{2m}( c^{(p_2)})^{\sigma} \bigg).
    \end{align*}
Consider now the vector $y\in\mathbb R^n$ defined as
\begin{align*}
    y\eqdef\big(\frac{x_1}{( \gamma^{(1)})^\sigma},x_2-\sum_{\nu=1} ^{n+1} \frac{\sigma_\nu}{ \gamma_\nu ^{(1)}},\ldots,\frac{x_{2p_1-1}}{( \gamma^{(p_1)})^\sigma},x_{2p_1}-\sum_{\nu=1} ^{n+1} \frac{\sigma_\nu}{ \gamma_\nu ^{(p_1)}},\frac{x_{2p_1+1}}{( c^{(1)})^\sigma},\ldots,\frac{x_n}{ ( c^{(p_2)})^\sigma}\big).
\end{align*}
Setting
\begin{align*}
	V_1\eqdef \big(0,\sum_{\nu=1} ^{n+1} \frac{\sigma_\nu}{ \gamma_\nu ^{(1)}},\ldots,0,\sum_{\nu=1} ^{n+1} \frac{\sigma_\nu}{ \gamma_\nu ^{(p_1)}},0,0,\ldots,0\big)
\end{align*}
and
\begin{align*}
	V_2 \eqdef \big((\gamma ^{(1)})^\sigma ,1,\ldots,(\gamma ^{(p_1)})^\sigma ,1,(c^{(1)})^\sigma,\ldots,(c^{(p_2)})^\sigma \big),
\end{align*}
we see that
\begin{align}\label{e.inner}
	V(2m+\sigma,\Gamma,C)=V_2\circ V(2m,\Gamma,C)+V_1\quad\mbox{and}\quad x=V_2\circ y+V_1.
\end{align}
Here we denote by $u\circ v$ the \emph{Hadamard product} of $u,v\in\mathbb R^n$: If $u=(u_1,\ldots,u_n)$, $v=(v_1,\ldots,v_n)$ then
 $$u\circ v=(u_1v_1,u_2v_2,\ldots,u_nv_n)\in\mathbb R^n.$$
For any $1\leq j \leq p_1$ we have that
\begin{align*}
    y_{2j-1}& =\frac{x_j}{( \gamma_1 ^{(j)})^{\sigma_1}\cdots ( \gamma_{n+1} ^{(j)})^{\sigma_{n+1}}}\\
&= \frac{x_j}{( a_1 ^{(j)})^{\sigma_1}\cdots ( a_{j-1} ^{(j)})^{\sigma_{j-1}} ( -a_{j} ^{(j)})^{\sigma_{j}}( a_{j+1} ^{(j+1)})^{\sigma_{j+1}}\cdots (- a_{n+1} ^{(j)})^{\sigma_{n+1}}}\\
&=\frac{(-1)^{\sigma_j}x_j}{( a_1 ^{(j)})^{\sigma_1}\cdots ( a_{n+1} ^{(j)})^{\sigma_{n+1}}}=\frac{|x_j|}{( a_1 ^{(j)})^{\sigma_1}\cdots ( a_{n+1} ^{(j)})^{\sigma_{n+1}}}>0.
\end{align*}
Similarly we can see that $y_j>0$ for all $2p_1+1\leq j \leq n$. This shows that $y\in (\mathbb R^+ \times\mathbb R)^{p_1}\times \mathbb (\mathbb R^+)^{p_2}$ and thus there is a sequence $m^{(\tau)}$ such that $V(2m^{(\tau)},\Gamma,C)\rightarrow y$ as $\tau\rightarrow \infty$. By \eqref{e.inner} we have that
\begin{align}
\lim_{\tau\to+\infty} V(2m^{(\tau)}+\sigma,\Gamma,C)=\lim_{\tau\to+\infty} V_2\circ V(2m^{(\tau)},\Gamma,C)+V_1= V_2 \circ y+V_1=x.
\end{align}
Since $x\in\mathbb R^n$ was arbitrary this concludes the proof of the lemma.
\end{proof}

Lemma \ref{l.positive} implies that in order to show that the $(n+1)$-tuple $T$ is hypercyclic it is enough to show that $\{V(2m,\Gamma,C):m\in\mathbb N_0 ^{n+1}\}$ is dense in $(\mathbb R^+ \times\mathbb R)^{p_1}\times \mathbb (\mathbb R^+)^{p_2}$, where $\Gamma$ and $C$ are defined by equations \eqref{e.gamma} and \eqref{e.c} respectively. We can reformulate this to get a linear condition in $m\in\mathbb N_0 ^{n+1}$, like in Corollary \ref{c.<3hyper}. Indeed, observe that the matrix $L$ now becomes
\begin{eqnarray*}
    L^+=
    \begin{pmatrix}
        \log a_1 ^{(1)} & \log  a_2 ^{(1)} & \log a_3 ^{(1)}&\cdots & \log a_{p_1-1} ^{(1)} & \log a_{p_1} ^{(1)}& \log a_{p_1+1} ^{(1)}& \cdots &\log  a_{n+1}^{(1)} \\
        {-1}/{ a_1 ^{(1)}} & {1}/{ a_2 ^{(1)}} & {1}/{ a_3 ^{(1)}} &\cdots& 1/ a_{p_1-1} ^{(1)}& 1/ a_{p_1} ^{(1)}& 1/ a^{(1)}  _{p_1+1}&\cdots & {-1}/{ a_{n+1} ^{(1)}} \\
        \log a_1 ^{(2)} & \log  a_2 ^{(2)} & \log a_3 ^{(2)}&\cdots& \log a_{p_1-1} ^{(2)} & \log a_{p_1} ^{(2)}& \log a_{p_1+1} ^{(2)}&\cdots &\log a_{n+1}^{(2)} \\
        {1}/{ a_1 ^{(2)}} & {-1}/{ a_2 ^{(2)}} &  {1}/{ a_3 ^{(2)}} &\cdots &1/ a_{p_1-1} ^{(2)} & 1/ a_{p_1} ^{(2)}& 1/ a_{p_1+1} ^{(2)}&\cdots & {-1}/{ a_{n+1} ^{(2)}} \\
        \vdots & \vdots & \vdots &\ddots &\vdots & \vdots & \vdots&\ddots&\vdots\\
        \log a_1 ^{(p_1)} & \log a_2 ^{(p_1)} & \log a_3 ^{(p_1)}    & \cdots& \log a_{p_1-1} ^{(p_1)} & \log a_{p_1} ^{(p_1)}& \log a_{p_1+1} ^{(p_1)}&\cdots &\log a_{n+1} ^{(p_1)} \\
        {1}/{a_1 ^{(p_1)}} & {1}/{a_2 ^{(p_1)}} &{1}/{a_3 ^{(p_1)}}& \cdots& 1/ a_{p_1-1} ^{(p_1)} & -1/ a_{p_1} ^{(p_1)}& 1/ a_{p_1+1} ^{(p_1)}&\cdots & {-1}/{a_{n+1} ^{(p_1)}} \\
        \log \delta_1 ^{(1)} & \log \delta_2 ^{(1)} &\log \delta_3 ^{(1)}&\cdots& \log \delta_{p_1-1} ^{(1)} &\log \delta_{p_1} ^{(1)} &\log \delta_{p_1+1} ^{(1)} & \cdots &\log \delta_{n+1} ^{(1)} \\
        \vdots & \vdots & \vdots &\ddots &\vdots & \vdots & \vdots&\ddots&\vdots\\
        \log \delta_1 ^{(p_2)} & \log \delta_2 ^{(p_2)} &\log \delta_3 ^{(p_2)} &\cdots&\log \delta_{p_1-1} ^{(p_2)}&\log \delta_{p_1} ^{(p_2)}&\log \delta_{p_1+1} ^{(p_2)}&\cdots &\log \delta_{n+1} ^{(p_2)}
    \end{pmatrix}
    ,
\end{eqnarray*}
where $a^{(b)}$ and $\delta^{(b)}$ have all their entries positive. We then have the desired inverse of Corollary \ref{c.<3hyper}:
\begin{proposition}\label{p.reduce}
    Suppose that the set $\{L^+m^T:m\in\mathbb N_0 ^{n+1}\}$ is dense in $\mathbb R^n$. Then $T$ is hypercyclic.
\end{proposition}
\begin{proof} Indeed, assuming that the set $\{L^+m^T:m\in\mathbb N_0 ^{n+1}\}$ is dense in $\mathbb R^n$ we immediately conclude that the set $\{V(2m,\Gamma,C):m\in\mathbb N_0 ^{n+1}\}$ is dense in $(\mathbb R^+ \times\mathbb R)^{p_1}\times \mathbb (\mathbb R^+)^{p_2}$. To see this note that, for any $1\leq b \leq p_1$, the set
    \begin{align*}
        \big\{2\sum_{\nu=1} ^{n+1} m_\nu \log a_\nu ^{(b)},\quad (m_1,\ldots,m_{n+1})\in\mathbb N_0 ^{n+1}\big\},
    \end{align*}
is dense in $\mathbb R$ if and only if the set
\begin{align*}
    \big\{(a_1^{(b)})^{2m_1}\cdots (a_{n+1}^{(b)})^{2m_{n+1}},\quad (m_1,\ldots,m_{n+1})\in\mathbb N_0 ^{n+1}\big\},
\end{align*}
is dense in $\mathbb R^+$. However this is the same as saying that the set
\begin{align*}
    \big\{(\gamma_1^{(b)})^{2m_1}\cdots (\gamma_{n+1}^{(b)})^{2m_{n+1}},\quad (m_1,\ldots,m_{n+1})\in\mathbb N_0 ^{n+1}\big\},
\end{align*}
is dense in $\mathbb R^+$ since $|\gamma_\nu ^{(b)}|=a_\nu ^{(b)}>0$ for all choices of $\nu$ and $b$. We reason similarly for the $c_\nu ^{(b)}$'s for  $1\leq b \leq p_2$. However, by Lemma \ref{l.positive} this implies that the set $\{V(m,\Gamma,C):m\in\mathbb N_0 ^{n+1}\}$ is dense in $\mathbb R^n$. By Lemma \ref{l.<3hyper} we then get that $T$ is hypercyclic.
\end{proof}
We will now construct the matrix $L^+$ so that $\overline{\{L^+m^T:m\in\mathbb N_0 ^{n+1}\}}=\mathbb R^n$. To that end it will be helpful to consider the $n+1$ vectors $u_1,\ldots,u_{n+1}\in\mathbb R^n$ which are just the corresponding columns of the $n\times (n+1) $ matrix $L^+$. That is we have:
\begin{align}
    \notag u_1 &\eqdef \bigg(\log a_1 ^{(1)},-\frac{1}{a_1 ^{(1)}},\log a_1 ^{(2)},\frac{1}{a_1 ^{(2)}},\log a _1 ^{(3)},\frac{1}{a_1 ^{(3)}},\ldots,\log a_1 ^{(p_1)},\frac{1}{a_1 ^{(p_1)}},\log\delta_1 ^{(1)},\ldots,\log \delta_1 ^{(p_1)}\bigg), \\
\notag   u_2 &\eqdef \bigg(\log a_2 ^{(1)},\frac{1}{a_2 ^{(1)}},\log a_2 ^{(2)},-\frac{1}{a_2 ^{(2)}},\log a _2 ^{(3)},\frac{1}{a_2 ^{(3)}},\ldots,\log a_2 ^{(p_1)},\frac{1}{a_2 ^{(p_1)}},\log\delta_2 ^{(1)},\ldots,\log \delta_2 ^{(p_1)}\bigg),\\ &\vdots \label{e.u1}\\ \notag
    u_{p_1} &\eqdef \bigg(\log a_{p_1} ^{(1)},\frac{1}{a_{p_1} ^{(1)}},\log a_{p_1}^{(2)},\frac{1}{a_{p_1} ^{(2)}},\log a _{p_1} ^{(3)},\frac{1}{a_{p_1} ^{(3)}},\ldots,\log a_{p_1} ^{(p_1)},-\frac{1}{a_{p_1} ^{(p_1)}},\log\delta_{p_1} ^{(1)},\ldots,\log \delta_{p_1} ^{(p_1)}\bigg),\\ &\vdots \notag \\ \notag
    u_{n+1} &\eqdef \bigg(\log a_{n+1} ^{(1)},-\frac{1}{a_{n+1} ^{(1)}},\log a_2 ^{(2)},-\frac{1}{a_{n+1} ^{(2)}},\ldots,\log a_{n+1} ^{(p_1)},-\frac{1}{a_{n+1} ^{(p_1)}},\log\delta_{n+1} ^{(1)},\ldots,\log \delta_{n+1} ^{(p_1)}\bigg).\\
\end{align}

The heart of the proof is the following theorem:
\begin{theorem}
    \label{t.vectorsu} For a positive integer $n\geq 3$, let $p_1\geq 1$ and $p_2\geq 0$ be non-negative integers such that $n=2p_1+p_2$. Then there exist real numbers $a_\nu ^{(b)}>0$, $1\leq b\leq p_1$, $1\leq \nu\leq n+1$ and $\delta_\nu^{(b)}>0$, $1\leq b \leq p_2$, $1\leq \nu \leq n+1$, such that, if we define the vectors $u_1,u_2,\ldots , u_{n+1}$ by \eqref{e.u1}, the following conditions are satisfied:
    \begin{list}
        {}{}
        \item{(i)} The vectors $u_1,\ldots,u_{n}$ are linearly independent over $\mathbb R$.
        \item{(ii)} There exist positive \textbf{irrational} numbers $c_1,\ldots,c_n$ such that $u_{n+1}=-\sum_{\nu=1} ^{n} c_\nu u_\nu$.
    \end{list}
    We conclude that
    \begin{list}
        {}{}
        \item{(iii)} Any $n$ of the vectors $u_1,\ldots,u_n,u_{n+1}$ are $\mathbb R$-linearly independent.
        \item{(iv)} The vectors $u_1,\ldots,u_n,u_{n+1}$ are $\mathbb Q$-linearly independent.
    \end{list}
\end{theorem}
Before giving the proof, let us see how we can use Theorem \ref{t.vectorsu} in order to prove Theorem \ref{t.njordan}.
\begin{proof}
    [Proof of Theorem \ref{t.njordan}] Because of Proposition \ref{p.reduce}, the proof of Theorem \ref{t.njordan} reduces to showing that the set $\{L^+m^T:m\in\mathbb N_0 ^{n+1}\}$ is dense in $\mathbb R^{n+1}$. By the definition of the vectors $u_1,\ldots,u_{n+1}$ this is equivalent to showing that the set
\begin{align*}
	\{L^+m^T:m\in\mathbb N_0 ^{n+1}\}=\{ m_1u_1+\cdots+m_{n+1}u_{n+1}:(m_1,\ldots,m_{n+1})\in\mathbb N_0 ^{n+1}\},
\end{align*}
is dense in $\mathbb R^n$. That is, we need to show that any $x\in\mathbb R^n$ can be approximated by linear combinations of the vectors $u_1,\ldots,u_{n+1}$ with coefficients in $\mathbb N$. To that end, we fix a $x\in\mathbb R^n$ and $\epsilon >0$. We write $x$ in the form
    \begin{align}
        x=R_1 u_1+\cdots R_n u_n+r_1 u_1\cdots+r_n u_n,
    \end{align}
    where $R_1,\ldots,R_n\in\mathbb Z$ and $r_1,\ldots,r_n\in [0,1)$. From Kronecker's theorem (see for example \cite{HW}) and (iv) of Theorem \ref{t.vectorsu} it follows that the sequence $\mathbb N u_{n+1}$ is dense in $\mathbb R^n/(\mathbb Z u_1+\cdots+\mathbb Z u_{n})$. So, we can find arbitrarily large $\ell\in\mathbb N$ such that
    \begin{align}
        \ell u_{n+1}=R'_1 u_1+\cdots+R'_{n}u_n+r'_1 u_1+\cdots+r'_{n}u_{n},
    \end{align}
    where $R'_\nu\in\mathbb Z$ and $r'_\nu\in(0,1)$ for all $1\leq\nu\leq n$, and
    \begin{align}
        |r_\nu-r'_\nu|<\frac{\epsilon}{\sum_{\nu=1} ^{n}\|u_\nu\|} , \quad \mbox{for all} \quad 1\leq \nu \leq n.
    \end{align}
    Now condition (ii) of Theorem \ref{t.vectorsu} implies that $R'_\nu<0$ for all $1\leq \nu \leq n$. In fact we can make the coefficients $R'_\nu$ as negative as we please by taking larger values of $\ell$. Let us now write
    \begin{align}
        x'\eqdef \ell u_{n+1}+ (R_1-R'_1)u_1+\cdots+(R_{n}-R'_{n})u_{n},
    \end{align}
    where we make sure that the coefficients $R_\nu-R'_\nu> 0$ for all $1\leq \nu \leq n$ by taking $\ell \in\mathbb N $ as large as necessary. We then have
    \begin{align}
        \|x-x'\|=\NOrm \sum_{\nu=1} ^{n}(r_\nu-r' _\nu) u_\nu. .\leq \sum_{\nu=1} ^{n} |r_\nu-r'_\nu| \ \|u_\nu\|<\epsilon.
    \end{align}
    Since $x'$ is a linear combination of $u_1,\ldots,u_{n+1}$ with coefficients in $\mathbb N$ we are done.
\end{proof}

In order to organize the proof of Theorem~\ref{t.vectorsu} we need two additional technical lemmas.
\begin{lemma}
    \label{l.nonlinear}Let $\delta_1,\delta_2>0$. For any $c>0$ the non-linear equation
    \begin{align*}
        x^{c+1}-\delta_1 x - \delta_2 c=0,
    \end{align*}
    has a unique positive solution $x=x(c)$. We have that $\lim_{c\rightarrow +\infty} x(c) = 1$.
\end{lemma}
\begin{proof}
    First observe that the function $f(x)=x^{c+1}-\delta_1 x - \delta_2 c$ is continuously differentiable in $x\in\mathbb R^+$ and satisfies $f(\delta_1^\frac{1}{c})=-\delta_2 c<0$ and $f(x)>0$ for $x$ large. Thus there is at least one $x_o\in\mathbb R^+$, $x_o>(\delta_1)^\frac{1}{c}$ such that $f(x_o)=0$. Looking at the derivative of $f$, $f'(x)=(c+1)x^c-\delta_1$ we see that $f$ has exactly one critical point at $x_1=\big(\frac{\delta_1}{c+1}\big)^\frac{1}{c}<x_o$. The function $f$ is negative for $0<x\leq x_1$ and strictly increasing for $x>x_1$ thus the solution $x_o$ is unique. We can define then the function $x(c)$ to be this unique solution.

    In order to prove that the function $x=x(c)$ has a limit as $c\rightarrow +\infty$ we argue as follows. First observe that for any $c>0$ we have that $f((\delta_2c)^\frac{1}{c+1})=-\delta_1 (\delta_2 c)^\frac{1}{c+1}<0$. On the other hand, for any $A>\delta_2$ we have that $f((Ac)^\frac{1}{c})=(Ac)^\frac{1}{c}(A-\frac{\delta_2}{(Ac)^\frac{1}{c}})c-\delta_1(Ac)^\frac{1}{c}\rightarrow +\infty$ as $c\rightarrow +\infty$. We thus see that for $c$ large enough, we have that $(\delta_2 c)^\frac{1}{c+1} < x(c) < (Ac)^\frac{1}{c}.$ Letting $c\rightarrow +\infty$ we conclude that $\lim_{c \rightarrow +\infty}x(c)=1$.
\end{proof}
In the following Lemma we give the basic construction which is `half-way there' to get the vectors $u_1,\ldots,u_{n+1}$ we need in Theorem \ref{t.vectorsu}. For this, the following notation will be useful. For any positive integer $n\geq 2$ and positive integers $p_1\geq 1$ and $p_2\geq 0$ such that $n=2p_1+p_2$, we consider the matrix $\tilde L(p_1,p_2)$ as follows
\begin{eqnarray*}
    \tilde L(p_1,p_2)= \begin{pmatrix}
        \log a_1 ^{(1)} & \frac {-1}{a_1 ^{(1)} } & \log a_1 ^{(2)} &   \frac{1}{a_1 ^{(2)}} & \cdots &\log a_1 ^{(p_1)} & \frac{1}{a_1 ^{(p_1)}}& \log \delta_1 ^{(1)}& \cdots & \log \delta_1 ^{(p_2)}  \\
        \log a_2 ^{(1)} & \frac{1}{ a_2 ^{(1)}} & \log a_2 ^{(2)} &   \frac{-1}{ a_2 ^{(2)}} & \cdots & \log  a_2^{(p_1)} & \frac{1}{ a_2^{(p_1)}} & \log \delta_2 ^{(1)} & \cdots & \log \delta_2 ^{(p_2)} \\
     \vdots &  \vdots & \vdots & \vdots & \ddots & \vdots & \vdots &\vdots &\ddots&\vdots \\
        \log a_{p_1} ^{(1)} & \frac{1}{ a_{p_1} ^{(1)}} & \log a_{p_1} ^{(2)} &  \frac{1}{ a_{p_1} ^{(2)}} &  \cdots & \log  a_{p_1}^{(p_1)} & \frac{-1}{ a_{p_1}^{(p_1)} }& \log \delta_{p_1} ^{(1)} & \cdots & \log \delta_{p_1} ^{(p_2)} \\
        \log a_{p_1+1} ^{(1)} &  \frac{1}{ a_{p_1+1} ^{(1)} }& \log a_{p_1+1} ^{(2)} &  \frac{1}{ a_{p_1+1} ^{(2)}} &  \cdots & \log  a_{p_1+1}^{(p_1)} & \frac{1}{ a_{p_1+1}^{(p_1)} }& \log \delta_{p_1+1} ^{(1)} & \cdots & \log \delta_{p_1+1} ^{(p_2)} \\
     \vdots &  \vdots & \vdots & \vdots & \ddots & \vdots & \vdots &\vdots &\ddots&\vdots \\
        \log a_{n-1} ^{(1)} & \frac{1}{ a_{n-1} ^{(1)}} & \log a_{n-1} ^{(2)} & \frac{1}{ a_{n-1} ^{(2)}} & \cdots &\log  a_{n-1}^{(p_1)} & \frac{1}{ a_{n-1}^{(p_1)}} & \log \delta_{n-1} ^{(1)} & \cdots & \log \delta_{n-1} ^{(p_2)} \\
        \log a_{n} ^{(1)} & \frac{1}{ a_{n} ^{(1)} }&\log a_{n} ^{(2)} & \frac{1}{ a_{n} ^{(2)}} & \cdots &\log  a_{n}^{(p_1)} & \frac{1}{ a_{n}^{(p_1)}} & \log \delta_{n} ^{(1)} &\cdots & \log \delta_{n} ^{(p_2)} \\
\end{pmatrix}
\end{eqnarray*}
where $a_\nu ^{(b)}>0$ for all $1\leq \nu \leq n$, $1\leq b \leq p_1$, and similarly $\delta_\nu ^{(b)}>0$ for $1\leq \nu \leq n$ and $1\leq b \leq p_2$. Observe that the value $p_2=0$ is allowed, in which case, there are no $\log \delta_\nu ^{(b)}$ terms. We also define the closely related matrix $\tilde L ^o(p_1,p_2)$ which is the special case of $\tilde L(p_1,p_2)$ if we set $a_{n} ^{(b)}=1$ for all $1\leq b \leq p_1$ and $\delta_{n} ^{(b)}=1$ for all $1\leq b \leq p_2$. That is we have
\begin{eqnarray*}
    \tilde L^o(p_1,p_2)\eqdef\begin{pmatrix}
        \log a_1 ^{(1)} & \frac {-1}{ a_1 ^{(1)} } & \log a_1 ^{(2)} &   \frac{1}{a_1 ^{(2)}} & \cdots &\log a_1 ^{(p_1)} & \frac{1}{a_1 ^{(p_1)}}& \log \delta_1 ^{(1)}& \cdots & \log \delta_1 ^{(p_2)}  \\
        \log a_2 ^{(1)} & \frac{1}{ a_2 ^{(1)}} & \log a_2 ^{(2)} &   \frac{-1}{ a_2 ^{(2)}} & \cdots & \log  a_2^{(p_1)} & \frac{1}{ a_2^{(p_1)}} & \log \delta_2 ^{(1)} & \cdots & \log \delta_2 ^{(p_2)} \\
     \vdots &  \vdots & \vdots & \vdots & \ddots & \vdots & \vdots &\vdots &\ddots&\vdots \\
        \log a_{p_1} ^{(1)} & \frac{1}{ a_{p_1} ^{(1)}} & \log a_{p_1} ^{(2)} &  \frac{1}{ a_{p_1} ^{(2)}} &  \cdots & \log  a_{p_1}^{(p_1)} & \frac{-1}{ a_{p_1}^{(p_1)} }& \log \delta_{p_1} ^{(1)} & \cdots & \log \delta_{p_1} ^{(p_2)} \\
        \log a_{p_1+1} ^{(1)} &  \frac{1}{ a_{p_1+1} ^{(1)} }& \log a_{p_1+1} ^{(2)} &  \frac{1}{ a_{p_1+1} ^{(2)}} &  \cdots & \log  a_{p_1+1}^{(p_1)} & \frac{1}{ a_{p_1+1}^{(p_1)} }& \log \delta_{p_1+1} ^{(1)} & \cdots & \log \delta_{p_1+1} ^{(p_2)} \\
     \vdots &  \vdots & \vdots & \vdots & \ddots & \vdots & \vdots &\vdots &\ddots&\vdots \\
        \log a_{n-1} ^{(1)} & \frac{1}{ a_{n-1} ^{(1)}} & \log a_{n-1} ^{(2)} & \frac{1}{ a_{n-1} ^{(2)}} & \cdots &\log  a_{n-1}^{(p_1)} & \frac{1}{ a_{n-1}^{(p_1)}} & \log \delta_{n-1} ^{(1)} & \cdots & \log \delta_{n-1} ^{(p_2)} \\
        0 & 1&0 & 1 &\cdots & 0 & 1 & 0 &\cdots &0 \\
\end{pmatrix}
\end{eqnarray*}

With this notation, we have the following Lemma:
\begin{lemma}
    \label{l.01constr} Let $n\geq 2$ be a positive integer and $p_1\geq 1 $ and $p_2\geq 0$ be non-negative integers such that $n=2p_1+p_2$. Then there exist $a_\nu ^{(b)}>0$ for $1\leq \nu \leq n-1$, $1\leq b \leq p_1$, and $\delta_\nu ^{(b)}>0$ for $1\leq \nu \leq n-1$, $1\leq b \leq p_2$ (if $p_2\neq 0$), such that $\det (\tilde L^o(p_1,p_2))\neq 0$.
\end{lemma}
\begin{proof}
    We will prove the Lemma by induction on $n$. For each such $n$ we have to consider all possible combinations of $p_1\geq 1$ and $p_2\geq 0$ such that $n=2p_1+p_2$, and this will be reflected in the inductive hypothesis. The first step of the induction is obvious. Indeed, for $n=2$ we necessarily have that $p_1=1$ and $p_2=0$. Then we need to show that there exists a choice of $a_1^{(1)}$ such that the matrix
    \begin{align*}
        \tilde L^o (1,0)=
        \begin{pmatrix}
            \log a_1^{(1)} & -1/ a_1^{(1)}\\
            0 &1,
        \end{pmatrix}
        ,
    \end{align*}
    has non-zero determinant. However this is the case for any $ a_1^{(1)}\neq 1.$

    Now assume the conclusion of the Lemma is true for $n-1$. There are two cases we need to consider depending on whether $p_2=0$ or $p_2\geq 1$. First we consider the case $n=2p_1+p_2$ where $p_1,p_2\geq 1$. Since $2p_1+(p_2-1)=n-1$ and $p_2-1\geq 0$, we can use the inductive hypothesis to get a matrix $\tilde L^o(p_1,p_2-1)$ with non-zero determinant. We need to construct $\tilde L^o(p_1,p_2)$. Let $ a_\nu ^{(b)}$, $1\leq \nu \leq n-2$, $1\leq b \leq p_1$ and $\delta_\nu ^{(b)}$, $1\leq \nu \leq n-2$, $1\leq b \leq p_2-1$, be defined as the corresponding entries of the matrix $\tilde L^o(p_1,p_2-1)$. We give arbitrary positive values to $\delta_\nu ^{(p_2)}$ for all $1\leq \nu \leq n-2$ as well as to $a_{n} ^{(b)}$ for all $1\leq b \leq p_1$ and to $\delta^{(b)}_{n}$ for $1\leq b \leq p_2-1$. Schematically we have
    \begin{eqnarray*}
        \tilde L^o(p_1,p_2)= \begin{pmatrix}
            \log a_1 ^{(1)} & \frac {-1}{ a_1 ^{(1)} } & \cdots &\log a_1 ^{(p_1)} & \frac{1}{a_1 ^{(p_1)}}& \log \delta_1 ^{(1)}& \cdots &\log \delta_1 ^{(p_2-1)}& \log * \\
            \log a_2 ^{(1)} & \frac{1}{ a_2 ^{(1)}} & \cdots & \log  a_2^{(p_1)} & \frac{1}{ a_2^{(p_1)}} & \log \delta_2 ^{(1)} & \cdots & \log \delta_2 ^{(p_2-1)}&\log *\\
             \vdots & \vdots & \ddots & \vdots & \vdots &\vdots &\ddots&\vdots&\vdots \\
            \log a_{p_1} ^{(1)} & \frac{1}{ a_{p_1} ^{(1)}} &  \cdots & \log  a_{p_1}^{(p_1)} & \frac{-1}{ a_{p_1}^{(p_1)} }& \log \delta_{p_1} ^{(1)} & \cdots & \log \delta_{p_1} ^{(p_2-1)}& \log *\\
            \log a_{p_1+1} ^{(1)} &  \frac{1}{ a_{p_1+1} ^{(1)} } &  \cdots & \log  a_{p_1+1}^{(p_1)} & \frac{1}{ a_{p_1+1}^{(p_1)} }& \log \delta_{p_1+1} ^{(1)} & \cdots &\log \delta_{p_1+1} ^{(p_2-1)} &\log *\\
             \vdots & \vdots & \ddots & \vdots & \vdots &\vdots &\ddots&\vdots&\vdots \\
            \log a_{n-2} ^{(1)} & \frac{1}{ a_{n-2} ^{(1)}} & \cdots &\log  a_{n-2}^{(p_1)} & \frac{1}{ a_{n-2}^{(p_1)}} & \log \delta_{n-2} ^{(1)} & \cdots &  \log \delta_{n-2} ^{(p_2-1)} & \log * \\
                \log * & 1/* & \cdots &\log * & 1/* & \log * & \cdots &  \log * & \log \delta_{n-1} ^{(p_2)} \\
            0 & 1&\cdots & 0 & 1 & 0 &\cdots &0&0 \\
    \end{pmatrix}
    \end{eqnarray*}
    where the wildcards `*' denote arbitrary positive choices. Now all the entries of $\tilde L^o(p_1,p_2)$ are defined except for $\log\delta_{n-1} ^{(p_2)}$ which we will now choose as follows. Developing the determinant of $\tilde L^o(p_1,p_2)$ with respect to the elements of the $(n-1)$-st row, we readily see that we can express it in the form
    \begin{align}
        \label{e.determinant} \det{ \tilde L^o(p_1,p_2)}=A+\log \delta_{n-1} ^{(p_2)} \det( \tilde L^o (p_1,p_2,-1)),
    \end{align}
    where $A$ is a constant that \emph{does not} depend on $\delta_{n-1} ^{(p_2)}$ but depends on all the other entries of the matrix which we have already fixed. Since $\det (\tilde L^o(p_1,p_2-1))\neq 0$ by the inductive hypothesis, there is always a choice of $\delta_{n-1} ^{(p_2)}$ such that the right hand side of \eqref{e.determinant} is non-zero so we are done in this case.

    Turning to the case $p_2=0$, that is $n=2p_1$, we need to construct the matrix $\tilde L^o(p_1,0)$ with the desired structure and non-zero determinant. Here the inductive hypothesis implies the existence of a matrix $\tilde L^o(p_1-1,1)$ with non zero determinant:
    \begin{eqnarray*}
        \tilde L^o(p_1-1,1)= \begin{pmatrix}
            \log a_1 ^{(1)} & \frac {-1}{ a_1 ^{(1)} } & \cdots &\log a_1 ^{(p_1-1)} & \frac{1}{a_1 ^{(p_1-1)}}& \log \delta_1 ^{(1)} \\
            \log a_2 ^{(1)} & \frac{1}{ a_2 ^{(1)}} & \cdots & \log  a_2^{(p_1-1)} & \frac{1}{ a_2^{(p_1-1)}} & \log \delta_2 ^{(1)} \\
             \vdots & \vdots & \ddots & \vdots & \vdots &\vdots  \\
            \log a_{p_1-1} ^{(1)} & \frac{1}{ a_{p_1-1} ^{(1)}} &  \cdots & \log  a_{p_1-1}^{(p_1-1)} & \frac{-1}{ a_{p_1-1}^{(p_1-1)} }& \log \delta_{p_1-1} ^{(1)} \\
            \log a_{p_1} ^{(1)} & \frac{1}{ a_{p_1} ^{(1)}} &  \cdots & \log  a_{p_1}^{(p_1-1)} & \frac{1}{ a_{p_1}^{(p_1-1)} }& \log \delta_{p_1} ^{(1)} \\
            \log a_{p_1+1} ^{(1)} &  \frac{1}{ a_{p_1+1} ^{(1)} } &  \cdots & \log  a_{p_1+1}^{(p_1-1)} & \frac{1}{ a_{p_1+1}^{(p_1-1)} }& \log \delta_{p_1+1} ^{(1)}\\
             \vdots & \vdots & \ddots & \vdots & \vdots &\vdots \\
            \log a_{n-2} ^{(1)} & \frac{1}{ a_{n-2} ^{(1)}} & \cdots &\log  a_{n-2}^{(p_1-1)} & \frac{1}{ a_{n-2}^{(p_1-1)}} & \log \delta_{n-2} ^{(1)}  \\
            0 & 1&\cdots & 0 & 1 & 0  \\
    \end{pmatrix}
    \end{eqnarray*}
    We define $a_\nu ^{(b)}$ as the corresponding entries of $\tilde L^o(p_1-1,1)$ for $1\leq \nu \leq n-2$ and $1\leq b \leq p_1-1$. We also define $a_\nu ^{(p_1)}=\delta_\nu ^{(1)}$ for $1\leq \nu \leq n-2$. Finally we give arbitrary positive values to $a_n ^{(b)}$ for all $1\leq b \leq p_1-1$. This defines all entries of $\tilde L^o(p_1-1,1)$ except the two rightmost entries on the $(n-1)$-st row. These are defined by the positive number $a_{n-1} ^{(p_2)}$ which we will choose now.  The matrix $\tilde L^o(p_1,0)$ has the following structure:
    \begin{eqnarray*}
        \tilde L^o(p_1,0)= \begin{pmatrix}
            \log a_1 ^{(1)} & \frac {-1}{ a_1 ^{(1)} } & \cdots &\log a_1 ^{(p_1-1)} & \frac{1}{a_1 ^{(p_1-1)}}& \log \delta_1 ^{(1)} &\frac{1}{\delta_1 ^{(1)}} \\
            \log a_2 ^{(1)} & \frac{1}{ a_2 ^{(1)}} & \cdots & \log  a_2^{(p_1-1)} & \frac{1}{ a_2^{(p_1-1)}} & \log \delta_2 ^{(1)}&\frac{1}{\delta_2 ^{(1)}} \\
             \vdots & \vdots & \ddots & \vdots & \vdots &\vdots &\vdots \\
            \log a_{p_1-1} ^{(1)} & \frac{1}{ a_{p_1-1} ^{(1)}} &  \cdots & \log  a_{p_1-1}^{(p_1-1)} & \frac{-1}{ a_{p_1-1}^{(p_1-1)} }& \log \delta_{p_1-1} ^{(1)}&\frac{1}{\delta_{p_1-1} ^{(1)}}\\
            \log a_{p_1} ^{(1)} & \frac{1}{ a_{p_1} ^{(1)}} &  \cdots & \log  a_{p_1}^{(p_1-1)} & \frac{1}{ a_{p_1}^{(p_1-1)} }& \log \delta_{p_1} ^{(1)}&\frac{-1}{\delta_{p_1} ^{(1)}} \\
            \log a_{p_1+1} ^{(1)} &  \frac{1}{ a_{p_1+1} ^{(1)} } &  \cdots & \log  a_{p_1+1}^{(p_1-1)} & \frac{1}{ a_{p_1+1}^{(p_1-1)} }& \log \delta_{p_1+1} ^{(1)}&\frac{1}{\delta_{p_1+1} ^{(1)}}\\
             \vdots & \vdots & \ddots & \vdots & \vdots &\vdots &\vdots\\
            \log a_{n-2} ^{(1)} & \frac{1}{ a_{n-2} ^{(1)}} & \cdots &\log  a_{n-2}^{(p_1-1)} & \frac{1}{ a_{n-2}^{(p_1-1)}} & \log \delta_{n-2} ^{(1)}  &\frac{1}{\delta_{n-2} ^{(1)}}\\
            \log * & 1/* &\cdots & \log * & 1/* & \log a_{n-1} ^{(p_1)} &{1}/{a_{n-1} ^{(p_1)}}  \\
            0 & 1&\cdots & 0 & 1 & 0 & 1 \\
    \end{pmatrix}
    \end{eqnarray*}
    where again the wildcards `*' denote arbitrary but fixed choices of positive real numbers. We develop the determinant of $\tilde L^o(p_1,0)$ with respect to the elements of the $(n-1)$-st row. We easily see that we have
    \begin{align}
        \label{e.determinant2} \det(\tilde L^o(p_1,0))=A+B\log a_{n-1} ^{(p_1)}+1/a_{n-1} ^{(p_1)} \det(\tilde L^o(p_1-1,1)),
    \end{align}
    where the constants $A, B$ are fixed real numbers that depend on all the entries of $\tilde L^o(p_1,0)$ except $a_{n-1} ^{(p_1)}$. Again, since $\det (\tilde L^o(p_1-1,1))\neq 0$ by the inductive hypothesis, no matter what the value of the constants $A, B$ is, there is always a choice of $a_{n-1} ^{(p_1)}>0$ that makes the right hand side of \eqref{e.determinant2} non-zero. This completes the proof.
\end{proof}
\begin{remark}
    We showed in Lemma \ref{l.01constr} that there is a choice of $a_\nu ^{(b)}$, $1\leq \nu\leq n-1$, $1\leq b \leq p_1$, and $\delta_\nu ^{(b)}$ for $1\leq \nu\leq n-1$, $1\leq b \leq p_2$, such that $\det (\tilde L^o(p_1,p_2))$ is non-zero. However, if we consider $\det (\tilde L^o(p_1,p_2))$ as a function of the entries $a_\nu ^{(b)}, \delta_\nu ^{(b)}$, we see this is a real-analytic function in $(\mathbb R^+)^{(p_1+p_2)(n-1)}$. We conclude that the set $\{\det (\tilde L^o(p_1,p_2))=0\}$ is a closed set which has zero $(p_1+p_2)(n-1)$-dimensional Lebesgue measure. See for example \cite{KR}. This means that for any dimension $n$ and any choice of $p_1\geq 1$ and $p_2\geq 0$ with $n=2p_1+p_2$, there are `generic' choices of $a_\nu ^{(b)}, \delta_\nu ^{(b)}$ such that the matrix $\tilde L^o(p_1,p_2) $ is invertible.
\end{remark}
\begin{proof}
    [Proof of Theorem \ref{t.vectorsu}] We will prove the Theorem based on the construction of Lemma \ref{l.01constr}. The conclusions (iii) and (iv) of the Theorem are easy consequences of (i) and (ii) so we will accept them with no further comment. Let $n$ and $p_1\geq 1$, $p_2\geq 0$, be given non-negative integers such that $n=2p_1+p_2$. We need to define the vectors $u_1,u_2,\ldots,u_n,u_{n+1}\in\mathbb R^n$  of the form \eqref{e.u1}, with $u_1,\ldots,u_n$ linearly independent over $\mathbb R$, as well as positive irrational constants $c_1,\ldots,c_n$ such that
    \begin{align}
        \label{e.linear} u_{n+1}=-\sum_{\nu=1} ^n c_\nu u_\nu.
    \end{align}
    We will choose the vectors $u_1,\ldots,u_n$ to be the rows of an appropriately constructed matrix $\tilde L(p_1,p_2)$. To that end, we consider the matrix $\tilde L(p_1,p_2)$ whose first $n-1$ rows are defined as the corresponding $n-1$ rows of $\tilde L^o(p_1,p_2)$ provided by Lemma \ref{l.01constr}. We first give arbitrary positive irrational values to the constants $c_1,\ldots,c_{p_1}$. Then we choose the constants $c_{p_1+1},\ldots,c_{n-1}$ in $\mathbb R^+\setminus \mathbb Q$ so that
    \begin{align}\label{e.auxpositive}
        \sum_{\nu=p_1+1} ^{n-1}\frac{c_\nu}{a_\nu ^{(b)}} - \sum_{\nu=1} ^{p_1} \frac{c_\nu}{a_\nu ^{(b)}}> 0,
    \end{align}
for all $1\leq b \leq p_1$. Observe that this is always possible when $p_1+p_2\geq 2$ since there are at least $n-1-(p_1+1)+1=p_1+p_2-1\geq 1$ `free choices' in the first term of the above expression. Actually this is the only point where we have to consider $n\geq 3$. The value of $c_n$ we leave undetermined for now.

We define the auxiliary variables
    \begin{align*}
      \frac{1} {\tilde a_{n+1} ^{(b)} }\eqdef \sum_{\substack{\nu=1\\\nu\neq b }} ^{n-1}\frac{ c_\nu}{a_\nu ^{(b)}}-\frac{c_b}{a_b ^{(b)}}, \quad 1\leq b \leq p_1,
    \end{align*}
    and
    \begin{align}
        \label{e.mainhat} \log \frac{1}{\hat a_{n+1} ^{(b)}} \eqdef \sum_{\nu=1} ^{n-1} c_\nu \log a_\nu ^{(b)}.
    \end{align}
    Similarly, if $p_2\neq 0$ we define
    \begin{align*}
        \frac{1}{\tilde \delta _{n+1} ^{(b)}} \eqdef \sum_{\nu=1} ^{n-1} \frac{c_\nu} {\delta_\nu ^{(b)}}, \quad 1\leq b \leq p_2,
    \end{align*}
    and
    \begin{align*}
        \log \frac{1}{\hat \delta_{n+1} ^{(b)} }\eqdef \sum_{\nu=1} ^{n-1} c_\nu \log \delta_\nu ^{(b)}, \quad 1\leq b \leq p_2.
    \end{align*}
    Observe that by definition $\hat a_{n+1}^{(b)}>0$ for all $1\leq b \leq p_1$ and $\tilde \delta_{n+1} ^{(b)},\hat \delta_{n+1} ^{(b)}>0$ for all $1\leq b \leq p_2$. We also have that
    \begin{align*}
    \frac{1}{   \tilde a_{n+1}^{(b)}}=\sum_{\substack{\nu=1\\\nu\neq b }} ^{n-1} \frac{c_\nu}{a_\nu ^{(b)}}-\frac{c_b}{a_b ^{(b)}}\geq \sum_{\nu=p_1+1} ^{n-1} \frac{c_\nu}{a_\nu ^{(b)}}-\sum_{\nu=1} ^{p_1}\frac{c_\nu}{a_\nu ^{(b)}}>0,
    \end{align*}
by \eqref{e.auxpositive}.

     Now we define $u_{n}$. For $c_{n}>0$ and for each $1\leq b\leq p_1$, we define $a_{n} ^{(b)}=a_n ^{(b)}(c_n)$ to be the unique positive solution of the non-linear equation
    \begin{align}
        \label{e.main1} (a_{n} ^{(b)})^{c_{n}+1}-\frac{\hat a_{n+1} ^{(b)}}{\tilde a_{n+1} ^{(b)}} a_{n} ^{(b)}-\hat a_{n+1} ^{(b)} c_n=0, \quad 1\leq b\leq p_1,
    \end{align}
    as a function of $c_{n}$. Similarly, for each $1\leq b \leq p_2$ we define $\delta_{n} ^{(b)}=\delta_{n} ^{(b)}(c_n)$ to be the unique solution of the non-linear equation
    \begin{align}
        \label{e.main2} (\delta_{n} ^{(b)})^{c_{n}+1}-\frac{\hat \delta_1 ^{(b)}}{\tilde \delta_{n+1} ^{(b)}}\delta_{n} ^{(b)}- \hat \delta_{n+1} ^{(b)} c_{n} =0, \quad 1\leq b\leq p_2,
    \end{align}
    again as a function of $c_{n}$. In order to define these solutions, we rely on Lemma \ref{l.nonlinear}.

\begin{remark} A few words are necessary to justify the definitions \eqref{e.main1} and \eqref{e.main2}. Observe, for example, that we need to define the numbers $c_n$ and $a_n ^{(b)}$ so that condition \eqref{e.linear} is satisfied. For any $1\leq b\leq p_1$, \eqref{e.linear} reads:
    \begin{align}\label{e.mot1}
        \log \frac{1}{a_{n+1} ^{(b)}}&=\sum_{\nu=1} ^n c_\nu \log a_\nu ^{(b)}=\sum_{\nu=1} ^{n-1} c_\nu \log a_\nu ^{(b)}+c_n \log a_n ^{(b)}=\log\frac{1}{\hat a ^{(b)} _{n+1}}+c_n\log a_n ^{(b)} \\ &= \log \frac{(a_n ^{(b)})^{c_n}}{\hat a_{n+1} ^{(b)}},
    \end{align}
and
    \begin{align}   \label{e.mot2}
        \frac{1}{a_{n+1} ^{(b)}}&=\sum_{\substack{\nu=1\\\nu\neq b }} ^{n-1} \frac{c_\nu}{a_\nu ^{(b)}} -\frac{c_b}{a_b ^{(b)}}+\frac{c_n}{a_n ^{(b)}}=\frac{1}{\tilde a_{n+1} ^{(b)}} +\frac{c_n}{a_n ^{(b)}}.
    \end{align}
Combining \eqref{e.mot1} and \eqref{e.mot2} we conclude that $c_n$ and $a_b ^{(b)}$ must satisfy \eqref{e.main1}. On the other hand, for $1\leq b \leq p_2$, equation \eqref{e.linear} gives
\begin{align}\label{e.mot3}
    \log \frac{1}{\delta_{n+1} ^{(b)}}= \frac{1}{\tilde \delta_{n+1} ^{(b)}}+\frac{c_n}{\delta_{n} ^{(b)}}
\end{align}
Technically speaking, it is enough to take arbitrary values for $\delta_n ^{(b)}$ and define $\delta_{n+1} ^{(b)}$ exactly by \eqref{e.mot3}. However, we prefer to define $\delta_{n+1} ^{(b)}$ again by the more restrictive non-linear equation \eqref{e.main2} for consistency.
\end{remark}

    For each $c_{n}>0$, the matrix $\tilde L(p_1,p_2)$ is defined in all its entries. We consider $\det (\tilde L(p_1,p_2))$ as a function of $c_{n}$. Suppose that we have $\det (\tilde L(p_1,p_2))(c_{n})=0$ for all irrational $c_{n}>0$. In this case there is a sequence of irrational $c_{n}\rightarrow +\infty$ for which we have that $\lim_{c_{n}\rightarrow +\infty}\det (\tilde L(p_1,p_2))(c_{n})=0$. However, according to Lemma \ref{l.nonlinear}, we have that 
\begin{align*}
	\lim_{c_n\rightarrow+\infty}a_n ^{(b)}(c_n)=1, \quad 1\leq b \leq p_1,
\end{align*}
and
\begin{align*}
	\lim_{c_n\rightarrow+\infty}\delta_n ^{(b)}(c_n)=1,\quad 1\leq b \leq p_2.
\end{align*}
But this means that
    \begin{align*}0=\lim_{c_{n}\rightarrow +\infty} \det (\tilde L(p_1,p_2)) (c_n)= \det(\tilde L^o(p_1,p_2)),\end{align*}
     which is a contradiction since we have chosen $\tilde L^o(p_1,p_2)$ so that its determinant is non-zero. We conclude that for large enough irrational and positive $c_{n}$, the matrix $\tilde L (p_1,p_2)$ has non-zero determinant. This means that its rows, the vectors $u_1,\ldots,u_{n}$, are linearly independent over $\mathbb R$.

    Finally, we define $a_{n+1} ^{(b)}$, $1\leq b \leq p_1$, exactly by the desired property \eqref{e.linear}:
    \begin{align}
        \label{e.cheapa} \frac{1}{a_{n+1} ^{(b)}} \eqdef \sum_{\substack{\nu=1\\\nu\neq b}} ^{n} \frac{c_\nu}{ a_\nu ^{(b)}}-\frac{c_b}{a_b ^{(b)}}=\sum_{\substack{\nu=1\\\nu\neq b}} ^{n-1} \frac{c_\nu }{a_\nu ^{(b)}}-\frac{c_b}{a_b ^{(b)}}+\frac{c_n}{a_n ^{(b)}}=\frac{1}{\tilde a_{n+1} ^{(b)}}+\frac{c_n}{a_n ^{(b)}}>0.
    \end{align}
Similarly, we define $\delta_{n+1} ^{(b)}$, $1\leq b \leq p_2$, as
    \begin{align}
        \label{e.cheapdelta} \frac{1}{\delta_{n+1} ^{(b)}}\eqdef \sum_{\nu=1} ^{n}\frac{ c_\nu}{ \delta_\nu ^{(b)}}.
    \end{align}
    It remains to check the validity of \eqref{e.linear}. We have for every $1\leq b \leq p_1$, that
    \begin{align*}
        \log \frac{1}{a_{n+1} ^{(b)}} &= \log \big(\frac{1}{\tilde a_{n+1} ^{(b)}} + \frac{c_{n}}{a_{n} ^{(b)}}\big) =\log \big( \frac{(a_{n+1} ^{(b)})^{c_{n+1}} }{\hat a_{n+1} ^{(b)}}\big),
    \end{align*}
    where in the last equality we have used the definition of $a_{n+1} ^{(b)}$, equation \eqref{e.main1}. Now, using the definition of $\hat a_1 ^{(b)}$ in equation \eqref{e.mainhat}, we get
    \begin{align*}
        \log a_{n+1} ^{(b)}=-c_{n+1} \log a_{n+1} ^{(b)}+\log \hat a_{n+1} ^{(b)}=-c_{n+1} \log a_{n+1} ^{(b)}-\sum_{\nu=2} ^{n+1} c_\nu \log a_\nu ^{(b)}=-\sum_{\nu=2} ^{n+1}c_\nu \log a_\nu ^{(b)}.
    \end{align*}
    A similar calculation for $1\leq b \leq p_2$ shows that $\log \delta_1 ^{(b)}=-\sum_{\nu=2} ^{n+1} c_\nu \log \delta_\nu ^{(b)}.$ This together with definitions \eqref{e.cheapa} and \eqref{e.cheapdelta} shows that $u_{n+1}=-\sum_{\nu=1} ^n c_\nu u_\nu$ and completes the proof of the Theorem.
\end{proof}
\subsection{The proof in the case $n=2$.} 
\label{sub.n=2} Here we have that $p_1=1$ and $p_2=0$. We define the vector
\begin{align}
    \gamma^{(1)}\eqdef (a_1 ^{(1)},a_2 ^{(1)},-a_3 ^{(1)}),
\end{align}
where $a_\nu ^{(1)}>0$ for $1\leq \nu\leq 3$. We define the triple of matrices $J=(J_1,J_2,J_3)$ by means of 
\begin{align*}
	J_\nu\eqdef \Jrd_{2,\gamma_\nu ^{(1)}}=\begin{pmatrix} \gamma_\nu ^{(1)} & 1 \\ 0 & \gamma_\nu ^{(1)}\end{pmatrix}, \quad 1\leq \nu \leq 3.
\end{align*}
With $\Gamma=\{\gamma_\nu ^{(1)}\}$, Lemma \ref{l.<3hyper} says that $J$ is hypercyclic if and only if $\overline{\{V(m,\Gamma):m\in\mathbb N^3\}}=\mathbb R^2$. Recall that for $m\in\mathbb N^3$
\begin{align*}
    V(m,\Gamma)=\big((\gamma^{(1)})^m,\sum_{\nu=1} ^3\frac{m_j}{\gamma_j ^{(1)}}\big).
\end{align*}
We have the analogue of Lemma \ref{l.positive} whose proof we omit.
\begin{lemma}\label{l.positive2}
    Suppose that the set $\{V(2m,\Gamma):m\in\mathbb N^3\}$ is dense in $\mathbb R^+\times \mathbb R$. Then $\{V(2m+1,\Gamma):m\in\mathbb N^3\}$ is dense in $\mathbb R^-\times \mathbb R$. We conclude that the set $\{V(m,\Gamma):m\in\mathbb N^3\}$ is dense in $\mathbb R^2$.
\end{lemma}
Consider the matrix $L^+$, \begin{align*}
    L^+=
    \begin{pmatrix}
        \log a_1 ^{(1)} & \log  a_2 ^{(1)} & \log a_3 ^{(1)} \\
        {1}/{ a_1 ^{(1)}} & {1}/{ a_2 ^{(1)}} & -{1}/{ a_3 ^{(1)}}
    \end{pmatrix}.
\end{align*}
Observe that in order to show that the set $V(2m,\Gamma)$ is dense in $\mathbb R^+\times\mathbb R$, it is enough to show that $\{L^+ m^T:m\in\mathbb N^3\}$ is dense in $\mathbb R^2$. We define the vectors $u_1,u_2,u_3$ as the corresponding columns of $L^+$,
\begin{align}\label{e.vectorsu2}
\notag  u_1 &\eqdef\big (\log a_1 ^{(1)},\frac{1}{a_1 ^{(1)}}\big),\\
    u_2 &\eqdef \big(\log a_2 ^{(1)},\frac{1 }{a_2 ^{(1)}}\big),\\
\notag  u_3 &\eqdef\big (\log a_3 ^{(1)},\frac{-1}{a_3 ^{(1)}}\big).
\end{align}
We have the analogue of Theorem \ref{t.vectorsu} which again is the main part of the proof.
\begin{theorem}\label{t.vectorsu2}
 There exist real numbers $a_\nu ^{(1)}>0$, $1\leq \nu\leq 3$ such that, if we define the vectors $u_1,u_2,u_3$ by \eqref{e.vectorsu2}, the following conditions are satisfied:
    \begin{list}
        {}{}
        \item{(i)} The vectors $u_1,u_2$ are linearly independent over $\mathbb R$.
        \item{(ii)} There exist positive \textbf{irrational} numbers $c_1,c_2$ such that $u_3=-(c_1u_1+c_2u_2)$.
    \end{list}
    We conclude that
    \begin{list}{}{}
        \item{(iii)} Any $2$ of the vectors $u_1,u_2,u_3$ are $\mathbb R$-linearly independent.
        \item{(iv)} The vectors $u_1,u_2,u_3$ are $\mathbb Q$-linearly independent.
    \end{list}
\end{theorem}
Using theorem \ref{t.vectorsu2}, we can conclude the proof of Theorem \ref{t.njordan} for $\mathbb R^2$ just like in the proof of the case $n\geq 3$. We will only describe how one proves Theorem \ref{t.vectorsu}, giving an argument very close to the one given in the proof of Theorem \ref{t.vectorsu}.
\begin{proof}[Proof of Theorem \ref{t.vectorsu2}] Consider the matrix
    \begin{align*}
        \tilde L^o (1,0)=
        \begin{pmatrix}
            \log a_1^{(1)} & 1/a_1^{(1)}\\
            0 &1,
        \end{pmatrix}
        ,
    \end{align*}
where $a_1 ^{(1)}\neq 1$. Obviously we then have $\det(\tilde L^o(1,0))\neq 0$. We give an arbitrary positive and irrational value to the constant $c_1$. Now for every $c_2>0$ we define $a_2 ^{(1)}=a_2 ^{(1)}(c_2)$ as the unique solution of the non-linear equation
\begin{align}\label{e.nonlinear2}
    (a_2 ^{(1)}) ^{c_2+1}-\frac{c_1}{(a_1 ^{(1)})^{c_1+1} } a_2 ^{(1)}- \frac{1}{(a_1 ^{(1)})^{c_1} } c_2 =0,
\end{align}
as a function of $c_2$. This is possible because of Lemma \ref{l.nonlinear}. The same Lemma also gives that $\lim_{\substack{c_2\rightarrow+\infty\\ c_2\in\mathbb R^+\setminus\mathbb Q}} a_2 ^{(1)}(c_2) = 1$. Therefore the matrix
\begin{align*}
\tilde L (1,0)=
        \begin{pmatrix}
            \log a_1^{(1)} & 1/a_1^{(1)}\\
            \log a_2 ^{(1)} &1/a_2 ^{(1)}
        \end{pmatrix},
\end{align*}
satisfies 
\begin{align*}
	\lim_{\substack{c_2\rightarrow+\infty\\ c_2\in\mathbb R^+\setminus\mathbb Q}} \det(\tilde L (1,0))=\det (\tilde L^o (1,0))\neq 0.
\end{align*} 
Choosing $c_2$ large enough in $\mathbb R^+\setminus \mathbb Q$ we get that the vectors $u_1,u_2$ are linearly independent over $\mathbb R$. Now we define
\begin{align*}
    \frac{1}{a_3 ^{(1)}} &\eqdef \frac{c_1}{a_1 ^{(1)}}+\frac{c_2}{a_2 ^{(1)}}.
\end{align*}
It remains to check the validity of $u_3=-(c_2u_2+c_3u_3)$. For this observe that
\begin{align*}
\log\frac{1}{a_3 ^{(1)}}=\log\big(\frac{c_1}{a_1 ^{(1)} }+ \frac{c_2 }{a_2 ^{(1)}}\big)=\log\big( \frac{(a_1 ^{(1)})^{c_1+1} (a_2 ^{(1)})^{c_2+1}}{a_1 ^{(1)} a_2 ^{(1)}}\big)=c_1\log a_1 ^{(1)}+c_2 \log a_2 ^{(1)},
\end{align*}
where the second equality is due to the non-linear equation \eqref{e.nonlinear2}.
\end{proof}

\bibliographystyle{amsplain}
\bibliography{jordan}

\end{document}